\documentclass[11pt]{amsart}
\usepackage{amsfonts,epsfig}
\usepackage{newlfont,amsfonts,amssymb,amsmath,amsthm,amsgen,amscd,datetime,enumerate}
\usepackage{multicol,picinpar}

\usepackage{mathrsfs}
\usepackage{euscript,color}

\usepackage{hyperref}
\hyperbaseurl{.}

\def\gg{\mathfrak{g}}

\def\hh{\mathfrak{h}}

\def\nn{\mathfrak{n}}
\def\ss{\mathfrak{s}}
\def\osc{\mathfrak{osc}}

\newcommand{\der}{{\mathfrak{der}}}

\newcommand{\p}{\mathfrak{p}}
\renewcommand{\t}{\mathfrak{t}}

\newcommand{\C}{\ensuremath{\mathbb{C}}}
\newcommand{\K}{\ensuremath{\mathbb{K}}}

\newcommand{\rr}{\ensuremath{\mathbb{R}}}
\renewcommand{\S}{\ensuremath{\mathbb{S}}}

\newcommand{\tr}{\mathrm{tr}}

\newcommand{\bmat}{\begin{pmatrix}}
\newcommand{\emat}{\end{pmatrix}}

\newcommand{\1}{\mathbf{1}}

\newcommand{\e}{\mathrm{e}}
\renewcommand{\d}{{\mathrm d}}
\newcommand{\bcase}{\begin{case}}
\newcommand{\ecase}{\end{case}}

\newcommand{\bclaim}{\begin{claim}}
\newcommand{\eclaim}{\end{claim}}

\newcommand{\bstep}{\begin{step}}
\newcommand{\estep}{\end{step}}

\newcommand{\bhlem}{\begin{hlem}}
\newcommand{\ehlem}{\end{hlem}}

	\newcommand{\pr}{\operatorname{pr}}

\newcommand{\bleer}{\begin{leer}}
\newcommand{\eleer}{\end{leer}}
\newcommand{\bde}{\begin{de}}
\newcommand{\ede}{\end{de}}

\newcommand{\bs}{\begin{satz}}
\newcommand{\es}{\end{satz}}
\newcommand{\btheo}{\begin{theo}}
\newcommand{\etheo}{\end{theo}}
\newcommand{\bfolg}{\begin{folg}}
\newcommand{\efolg}{\end{folg}}
\newcommand{\blem}{\begin{lem}}
\newcommand{\elem}{\end{lem}}
\newcommand{\bnote}{\begin{note}}
\newcommand{\enote}{\end{note}}
\newcommand{\bprf}{\begin{proof}}
\newcommand{\eprf}{\end{proof}}
\newcommand{\bd}{\begin{displaymath}}
\newcommand{\ed}{\end{displaymath}}
\newcommand{\be}{\begin{eqnarray*}}
\newcommand{\ee}{\end{eqnarray*}}
\newcommand{\eeqa}{\end{eqnarray}}
\newcommand{\beqa}{\begin{eqnarray}}
\newcommand{\bi}{\begin{itemize}}
\newcommand{\ei}{\end{itemize}}
\newcommand{\bnum}{\begin{enumerate}}
\newcommand{\enum}{\end{enumerate}}
\renewcommand{\la}{\langle}
\renewcommand{\ra}{\rangle}

\newcommand{\beq}{\begin{equation}}
\newcommand{\eeq}{\end{equation}}

\newcommand{\ccc}{\mathbb{C}}

\newcommand{\vf}{\varphi}
\newcommand{\earr}{\end{array}\]}
\newcommand{\barr}{\[\begin{array}}
\newcommand{\bvec}{\left(\begin{array}{c}}
\newcommand{\evec}{\end{array}\right)}

\newcommand{\g}{\mathrm{g}}
\newcommand{\hg}{\hat{\g}}
\newcommand{\Ric}{\mathrm{Ric}}
\newcommand{\R}{\mathrm{R}}
\newcommand{\A}{\mathrm{A}}
\newcommand{\B}{\mathrm{B}}
\newcommand{\W}{\mathrm{C}}

\newcommand{\h}{\mathrm{h}}

\newcommand{\n}{\mathfrak{n}}
\renewcommand{\k}{\mathfrak{k}}
\renewcommand{\sl}{\mathfrak{sl}}

\newcommand{\+}{\oplus}

\newcommand{\so}{\mathfrak{so}}
\renewcommand{\b}{\mathrm{b}}

\newcommand{\s}{\sigma}

\newcommand{\bbem}{\begin{bem}}
\newcommand{\ebem}{\end{bem}}
\newcommand{\bbez}{\begin{bez}}
\newcommand{\ebez}{\end{bez}}
\newcommand{\bbsp}{\begin{bsp}}
\newcommand{\ebsp}{\end{bsp}}

\newcommand{\tg}{\widetilde{\g}}

\newcommand{\ro}{\mathsf{P}}
\newcommand{\rot}{\mathsf{J}}

\newcommand{\ann}{\mathfrak{ann}}

\newcommand{\Ad}{\mathrm{Ad}}
\newcommand{\ad}{\mathrm{ad}}

\renewcommand{\div}{\mathrm{div}}

\newcommand{\belabel}[1]{\begin{equation}\label{#1}}

\theoremstyle{plain}
\newtheorem{theorem}{Theorem}[section]
\newtheorem{lemma}{Lemma}[section] 
\newtheorem*{lem*}{Lemma}
\newtheorem{proposition}{Proposition} [section]
\newtheorem{corollary}{Corollary}[section] 

\theoremstyle{definition}
\newtheorem{definition}{Definition}[section] 
\newtheorem{remark}{Remark}[section] 
\newtheorem{bez}{Notation}[section] 
\newtheorem{example}{Example}[section] 

\newtheorem*{bsp*}{Example}
\newtheorem*{def*}{Definition}

\usepackage{cleveref}

\newcommand{\hook}{\makebox[7pt]{\rule{6pt}{.3pt}\rule{.3pt}{5pt}}\,}
\renewcommand{\hat}{\widehat}

\usepackage{euscript,color}

\numberwithin{equation}{section}

\begin{document}
\bibliographystyle{abbrv}

\title
[Conformal properties of  indefinite bi-invariant metrics]
{Conformal properties of indefinite  bi-invariant metrics}
\author[Kelli Francis-Staite]{Kelli Francis-Staite}
\address[Francis-Staite]{Mathematics Department\\ The University of Oxford\\ OX2 6GG\\United Kingdom}
\email{kelli.francis-staite@maths.ox.ac.uk}
\author[Thomas Leistner]{Thomas Leistner}
\address[Leistner, corresponding author]{School of Mathematical Sciences\\University of Adelaide\\SA5005\\Australia
}
\email{thomas.leistner@adelaide.edu.au}

%

\thanks{The
second author acknowledges support from the Australian Research
Council via the grant FT110100429. The first author was funded by a Master of Philosophy scholarship from the University of Adelaide.
}

  \begin{abstract}
  An indecomposable Lie group with Riemannian bi-invariant metric is always simple and hence Einstein. For indefinite metrics this is no longer true, not even for simple Lie groups.
  We study the question of whether a semi-Riemannian bi-invariant metric is conformal to an Einstein metric. We obtain results for all three  cases in  the structure theorem by Medina and Revoy for indecomposable metric Lie algebras: the case of simple Lie algebras, and the cases  of 
    double extensions of  metric Lie algebras by $\rr$ or a simple Lie algebra. 
     Simple Lie algebras  are conformally Einstein precisely when they are Einstein, or when equal to $\sl_2\C$ and  conformally flat. Double extensions of metric Lie algebras by simple Lie algebras of rank greater than one are never conformally Einstein, and neither are double extensions of Lorentzian oscillator algebras, whereas the oscillator algebras themselves are conformally Einstein. Our results give a complete answer to the question of which metric Lie algebras in Lorentzian signature and in signature $ (2,n-2)$ are  conformally Einstein.
   %
%
  \end{abstract}

 \subjclass[2010]{Primary: 53C50, 53C35, 53A30; Secondary: 22E60}
\keywords{Bi-invariant metrics, conformal Einstein metrics, double extensions of Lie algebras}

\maketitle
  
   \tableofcontents
  \section{Introduction and statement of results}
  In this paper, we study the conformal properties of bi-invariant metrics on Lie groups by considering properties of their Lie algebras. Recall that a semi-Riemannian metric on a Lie group $G$ is called {\em bi-invariant} if all  multiplications from left  and right are isometries.  
  A Lie group with a bi-invariant semi-Riemannian metric  is called a {\em metric Lie group}. A~bi-invariant metric of signature $(p,q)$ on $G$ induces a scalar product of signature $(p,q)$ on the Lie algebra $\gg$ of $G$ that is  $\Ad$-invariant, i.e., invariant under the adjoint representation $\Ad$ of $G$ on $\gg$, and consequently invariant under its differential  $\ad$. In fact, on a connected Lie group, bi-invariant metrics are in 1-1 correspondence with  $\ad$-invariant scalar products on the Lie algebra. A Lie algebra $\gg$ with a scalar product (of signature $(p,q)$) is called a {\em metric Lie algebra (of signature $(p,q)$)}.
 
  In the following, when studying the geometry of a metric Lie group $G$ we will do this by studying metric Lie algebras, and when referring to geometric objects on $G$, such as  the curvature tensor, the Ricci tensor, etc., we will just refer to curvature tensor, Ricci tensor, etc.~of the metric Lie algebra $\gg$.
  
A metric Lie algebra is called {\em decomposable} if the Lie algebra is isomorphic to a direct sum of orthogonal ideals. If there is no such decomposition, we call the metric Lie algebra and its corresponding Lie group {\em indecomposable}. Due to splitting theorems of de Rham and Wu decomposability in this algebraic sense   is related to decomposability in the geometric sense, i.e., to the fact that the metric Lie group, if it is simply connected,  decomposes into a semi-Riemannian product manifold. Hence, indecomposable metric Lie algebras can be considered as the fundamental building blocks of metric Lie algebras. More surprising is the following striking structure result for indecomposable metric Lie algebras:
    \begin{theorem}[Medina \& Revoy \cite{medina-revoy85}]
  \label{mrtheo}
Every  indecomposable metric Lie algebra is either one-dimensional, simple, or a double extension of a metric Lie algebra $\hh\not=\{0\}$ by another  Lie algebra $\ss$ and a Lie algebra homomorphism $\delta:\ss\to \so(\hh)\cap\der(\hh)$ into to the skew derivations of $\hh$ such that:
    \begin{enumerate}[(a)]
    \item $\ss$ is simple or $\ss=\rr$, and 
        \item the image of $\delta$ is not contained in the inner derivations $\ad(\hh)$, \cite{FavreSantharoubane87, Figueroa-OFarriStanciu96}.
    \end{enumerate}
      For a double extension, the signature of the metric is $(p+\dim \ss, q+\dim\ss)$, where $(p,q)$ is the signature of the metric Lie algebra $\hh$. 
      \end{theorem}
      
     It is  known that indecomposable {\em Riemannian} metric Lie algebras are simple or one-dimensional; the signature description from Theorem \ref{mrtheo} confirms this. The interesting part of \Cref{mrtheo} however is the statement about indefinite metric Lie algebras and these will  be a focus of our paper.
   
 In the following we will study     conformal properties of metric Lie groups, such as (local) conformal flatness or being (locally) conformally equivalent to an Einstein space (for precise definitions see the following paragraph and \Cref{cesec}).  As we will formulate our results in terms of the associated metric Lie algebra, we leave aside the difficulties arising from the transition from local to global and the fact that there may be  several (locally isometric) metric Lie groups having the same metric Lie algebra. We will say that a metric Lie algebra $\gg$ has a conformal property if a Lie group with bi-invariant metric and metric Lie algebra $\gg$  has the corresponding local conformal property. 
For example, given the equivalence of local conformal flatness with the vanishing of the Weyl tensor, we say that a metric Lie algebra $\gg$ is {\em conformally flat} if its Weyl tensor vanishes. 
 
Our main focus is the property of a bi-invariant metric on a  Lie group  to be locally conformally equivalent to an Einstein metric (again see \Cref{cesec} for details). Of course,  the resulting metric is no longer bi-invariant (unless the scaling function is constant). In contrast to the locally conformally flat property, there is in general no tensorial condition that is equivalent  to the locally conformal Einstein property.
Instead, it is equivalent   to  the following  differential equation: a metric $\g$ on a manifold $M$ is locally conformally equivalent to an Einstein metric if and only if
\begin{enumerate}
\item[(A)] each point in $M$ has a neighbourhood $U$ with a closed $1$-form $\Upsilon\in \Gamma(T^*U)$ such that that the trace-free part of \[
\ro -\nabla \Upsilon +\Upsilon^2 \]
vanishes, where $\ro$ is the Schouten tensor and $\nabla$ the Levi-Civita connection of $\g$.
\end{enumerate}
For Lie groups with bi-invariant metric this property 
cannot be  be formulated purely in terms of the metric Lie algebra.  Instead we make the following definition:
\begin{definition}\label{cedef}
A metric Lie algebra $\gg$ is  {\em conformally Einstein}  if for the unique simply connected metric Lie group $(G,\g)$ with metric Lie algebra $\gg$  the  bi-invariant metric $\g$ is locally conformally equivalent to an Einstein metric\footnote{Of course, as we only require the local property in the definition, if the simply connected metric Lie group is locally conformally Einstein, then all other metric Lie groups with the same metric Lie algebra also satisfy this property.}.
\end{definition}

Even though in general  the locally conformally Einstein property is not fully characterised by a tensorial condition,  there are  certain {\em tensorial obstructions} for the original metric to be conformally Einstein. Remarkably, under some genericity conditions on the Weyl tensor, the vanishing of these obstructions  is  not only necessary but  also sufficient for the metric to be conformally Einstein. In general dimensions these obstructions were found by Gover and Nurowski \cite{gover-nurowski04}. They allow us to check effectively whether a given metric can be conformally Einstein by computing certain tensors, instead of attempting to solve the PDE in (A) directly. 
We will see that in the case of a metric Lie algebra $\gg$, the vanishing of these obstructions simplifies  to the following two conditions:
\begin{enumerate}
\item[(B)] The Bach tensor $\B$ of $\gg$ vanishes, and
\item[(C)] 
 if $n=\dim(\gg)>4$ and $\gg$ is not already Einstein, the {\em Weyl nullity ideal} $\n$ in $\gg$ is not zero. Here $\n$ is the ideal
\[\n=\{X\in \gg\mid X\hook \W=0\}\subset \gg, \]
where $\W$ is the Weyl tensor of  $\gg$.
\end{enumerate}

Note that the question whether an indecomposable {\em Riemannian}, and hence simple,  metric Lie algebra is conformally Einstein is trivial as it is already Einstein: the only candidate for an  $\ad$-invariant positive definite bilinear form is the Killing form,
and therefore the simple Lie algebra $\gg$ is of compact type and Einstein with positive Einstein constant  (see for example Milnor's classical paper \cite{milnor76}).

With the aim of determining which indefinite metric Lie algebras satisfy the necessary conditions (B) and (C),  we start by describing the Bach tensor of a metric Lie algebra and expressing its vanishing in terms of the Ricci tensor. In \Cref{liesec} we prove our first result:
\begin{theorem}\label{bachtheo}
%
%
A metric Lie algebra  of dimension $n>2$  is Bach flat if and only if it is Einstein, or its Ricci tensor $\Ric$ satisfies one of the following conditions:
\begin{enumerate}
\item $\Ric$  is  $2$-step nilpotent.
\item $\Ric$ 
is diagonalisable with two different eigenvalues $\lambda\not=0$ and $\mu =-\tfrac{n-k-1}{k-1}\lambda$, where $k$ is the dimension of the eigenspace of $\lambda$, and with non-degenerate eigenspaces. In this case, if $\Ric$ has a non trivial kernel, it is of dimension $1$ and non degenerate.
\end{enumerate}
\end{theorem}
By a $2$-step nilpotent Ricci tensor we mean that the endomorphism of $\gg$ that is obtained by dualising the Ricci tensor with the metric squares to zero, $\Ric^2=0$ (including the Ricci-flat case). 
As a corollary we obtain that solvable metric Lie algebras are Bach flat. 

Next, in \Cref{simplesec} we consider the case of {\em simple metric Lie algebras} $\gg$ in \Cref{mrtheo} by considering their complexifications $\gg^\C$. Extending the trivial Riemannian situation,  we show:
\begin{theorem}\label{simpletheoointro}
Let $\gg$ be a simple metric Lie algebra. If $\gg^\C$ is simple, the metric is given by the Killing form and hence Einstein. If $\gg^\C$ is not simple, then there is a $2$-parameter family of bi-invariant metrics of neutral signature $(\tfrac{n}{2},\tfrac{n}{2})$. Moreover:
\begin{enumerate}
\item The only Bach flat metrics in this family are the multiples of the Killing form  of $\g$ and of the imaginary part of the Killing form of $\g^\ccc$.
\item 
The only conformal Einstein metrics in this class are multiples of the Killing form of $\g$ (which are Einstein) and,  when $\gg=\sl_2\C$ (as real Lie algebra),  the multiples of the imaginary part of the Killing form of $\sl_2\C$ (which are conformally flat).
\end{enumerate}
\end{theorem}
Motivated by the other cases in \Cref{mrtheo}  we turn to metric Lie algebras that are given by {\em double extensions} (for definition and details see \Cref{desec}).
We obtain the following consequence of  \Cref{bachtheo}.
\begin{corollary}\label{theodextric}
Let $\gg$ be a metric Lie algebra given by a double extension as in \Cref{dedef}. Then 
\begin{enumerate}
\item $\gg$ is Einstein if and only if it is Ricci flat, and 
\item
$\gg$ is Bach-flat if and only if  $\Ric^2=0$.
\end{enumerate}
\end{corollary}
As another consequence to \Cref{bachtheo} we obtain that double extensions $\gg$ of nilpotent Lie algebras $\hh$ have $2$-step nilpotent Ricci tensor. 

Next we turn to the case in \Cref{mrtheo} where $\gg$ is the {\em double extension of a metric Lie algebra by a simple Lie algebra}. Using the the necessary conditions (B) and (C)  we  show in \Cref{desimplesec}:
\begin{theorem}
\label{desimpletheo}
If the double extension $\gg$ of a metric Lie algebra $\hh$ by a simple Lie algebra $\ss$ is  conformally Einstein, then $\ss=\sl_2\rr$ or $\ss=\so(3)$.
\end{theorem}
While we will leave the cases of $\ss=\sl_2\rr$ or $\ss=\so(3)$ undecided in general, in \Cref{so3ex} we define a double extension of $\hh=\rr^3$ by $\ss=\so(3)$ that satisfies both conditions  (B) and (C), but does not satisfy condition (A) and hence is {\em not conformally Einstein}. 

The remaining case in \Cref{mrtheo} of {\em double extensions $\gg$ by the $1$-dimensional Lie algebra $\ss=\rr$} is the most difficult one. It is however important as every indecomposable Lorentzian metric Lie algebra is either  isomorphic to $\sl_2\rr$, or to a   Lorentzian oscillator algebra, 
\cite{medina85}.  
In general, the {\em oscillator algebra $\osc_{\Phi}(t,s)$ of signature 
 $(t+1,s+1)$} is defined as the double extension of the abelian metric Lie algebra $\rr^{t,s}$ by $\rr$ and  $\Phi\in\so(t,s)$. 
 The Lorentzian oscillator algebras of dimension $n$ then are denoted by  $\osc_{\Phi}(n-2)=\osc_{\Phi} (0,n-2)$. All oscillator algebras are 
  solvable and hence Bach flat. 
 In Section \ref{oscsubsec}  we show:
 \begin{theorem}\label{osctheo}
The oscillator Lie algebras $\mathfrak{osc}_\Phi(t,s)$ are  conformally Einstein. In particular,  all indecomposable Lorentzian metric Lie algebras and all indecomposable $4$-dimensional  metric Lie algebras are  conformally Einstein. 
\end{theorem}
The statement about the Lorentzian metric Lie algebras in this theorem is already known from a more general result in \cite{leistner05a}, where it was shown that Lorentzian plane waves and hence Lorentzian symmetric spaces with solvable transvection group, the so-called {\em Cahen-Wallach spaces} \cite{cahen-wallach70}, are conformally Einstein. In fact, locally they admit {\em two different} rescalings to Ricci-flat metrics.

Finally we turn to indecomposable metric Lie algebras $\gg$ of  signature $(2,n-2)$, of which there are three nonsimple types \cite{baum-kath03}: 
the oscillator algebras $\gg=\osc_\Phi(1,n-3)$, which are conformally Einstein by \Cref{osctheo}, and the double extensions $\gg$ of $\hh$ by $\rr$ and a derivation of $\hh$, for  either $\hh=\osc_\Phi(n-4)$ or  $\hh=\osc_\Phi(n-5)\+\rr$. It turns out that both of the latter are Bach flat but are not conformally Einstein. This and  checking that no simple real Lie algebra 
has a Killing form of signature $(2,n-2)$ for $n\ge 4$ (for example in \cite[Chapter 15]{SagleWalde73}), yields our final result.

\begin{theorem}
\label{2ntheo}
The oscillator algebras $\osc_\Phi (1,n-3)$, with $\Phi\in\so(1,n-3)$, are the only indecomposable metric Lie algebras of signature $(2,n-2)$ that are  conformally Einstein.
\end{theorem}
The results in signature $(2,n-2)$ indicate that there might not be a general pattern for the conformally Einstein property in higher signature, however the methods and results presented here should allow one to decide in specific cases whether or not the property holds. 

\subsection*{Acknowledgements}Many results in this paper were obtained as part the first author's  Master thesis \cite{kelli-thesis}, which was written at the University of Adelaide under supervision of  Michael Murray and   the second author. We would like to thank Michael Murray for his support and helpful remarks. We also thank Vicente~Cort\'{e}s and Wolfgang~Globke for useful discussions on some aspects of the paper, and Michael Eastwood for crucial comments on \Cref{so3ex}.


  \section{Conformal Einstein metrics} \label{cesec}
Let $(M,\g)$ be a semi-Riemanian manifold. Our conventions regarding the Riemann {\em curvature tensor} $\R$ and the {\em Ricci curvature} $\Ric$ are as follows:
\begin{align*}
\R(X,Y)Z&=\left[\nabla_X,\nabla_Y\right]-\nabla_{[X,Y]},& \R(X,Y,Z,V)&=\g(\R(X,Y)Z,V),\\
 \Ric(X,Y)&=\tr(Z\mapsto\R(Z,X)Y).\end{align*}
We will frequently use index notation as {\em abstract indices} (in the sense of Penrose) as well as denoting the component of tensor with respect to a basis $X_i$, e.g.,
\[\R_{ijkl}=\R(X_i,X_j,X_k,X_l),\quad   \R_{ij}=\Ric(X_i,X_j)=\R_{kij}^{\ \ \ \ k}=\g^{kl}\R_{kijl}.\]
We will also dualise tensors with the metric, i.e., raising and lowering indices using the metric and its inverse. 
Moreover, we define the {\em Schouten tensor} $\ro$ by 
\[(n-2)\ro= \Ric -\frac{\rho}{2(n-1)}\g,
\]
where $\rho$ is the {\em scalar curvature}, and the {\em Weyl tensor} as 
\[\W_{abcd}=\R_{abcd}+2(\g_{a[c} \ro_{d]b}+ \g_{b[d}\ro_{c]a}),\]
where $T_{[ab]}=\frac{1}{2} (T_{ab}-T_{ba}) $ denotes the skew-symmetrisation of a tensor.
The trace of the Schouten tensor and the scalar curvature $\rho=\g^{bc}\R_{bc}$ are related by
\[\mathsf{J}=\g^{bc}\ro_{bc}=\frac{\rho}{2(n-1)}.\]
Finally, the {\em Cotton}  and {\em Bach tensors} are given by
\[\A_{abc}=\nabla_{[b}\ro_{c]a},\qquad \B_{bc}= \ro^{ad}\W_{abcd}  -\nabla^d\A_{bcd}.\]
Finally, a semi-Riemannian manifold $(M,\g)$ is {\em Einstein} and $\g$ is an {\em Einstein metric} if its Ricci tensor, or equivalently its Schouten tensor,  is a (possibly vanishing) multiple of the metric. If $\Ric=0$, the metric is {\em Ricci-flat}.

Given a smooth manifold $M$ we say that two semi-Riemannian metrics $\g$ and $\hg$ on $M$  are {\em conformally equivalent} if there is a smooth function $\vf$ on $M$ such that $\hg=\mathrm{e}^{2\vf}\g$, and that they are {\em locally conformally equivalent} if each point in $M$ has a neighbourhood $U$ such that on $U$ the metrics $\hg|_U$ and $\g|_U$ are conformally equivalent.
The Schouten tensors $\ro$ and $\hat \ro$ of $\g$ and $\hat \g$ and their traces $\mathsf{J}$ and $\hat{ \mathsf{J}}$ are then related as follows,
\begin{equation}
\label{rotrafo}
\hat \ro = \ro -\nabla \Upsilon +\Upsilon^2 -\tfrac{1}{2} \, \g(\Upsilon,\Upsilon)\, \g,\qquad
\hat {\mathsf{J}}=\mathrm{e}^{-2\vf}\left(\mathsf{J}-\mathrm{div}(\Upsilon)-\tfrac{n-2}{2}\g(\Upsilon,\Upsilon)\right),
\end{equation}
where $\Upsilon=\d\vf$ (see for example \cite{besse87}).
Moreover, a semi-Riemannian manifold is {\em (locally) conformally flat} if it (locally) conformally equivalent to the flat metric. It is well known since \cite[p. \textsection 28]{eisenhardt49} that local conformal flatness is equivalent to the vanishing of the Weyl tensor~$\W$. Note that, even for simply connected manifolds, $\W=0$ does not imply global conformal flatness, as the example of the sphere shows.

A semi-Riemannian manifold $(M,\g)$ is {\em (locally) conformally Einstein} if $\g$  is (locally) conformally equivalent to an Einstein metric.
 The transformation of the Schouten tensor under $\g\mapsto \hat \g$ in \cref{rotrafo} reveals that the metric $\hat\g=\e^{2\vf}\g$ is Einstein if and only if the function $\vf$ on $M$ satisfies the following PDE
 \belabel{cepde}
 \ro- \nabla^2\vf + (\d\vf)^2-\frac{\rot-\Delta(\vf)+\d\vf (\nabla \vf)}{n}\ \g=0,
 \end{equation}
where $\nabla \vf$ is the gradient, $\nabla^2\vf=\nabla \d\vf$ denotes the Hessian, $\Delta(\vf)=\div(\nabla \vf)$ the Laplacian of $\vf$ and $\rot$ is the trace of $\ro$, all with respect to $\g$.
Therefore, locally the following conditions are equivalent:
\begin{enumerate}
\item $(M,\g)$ is locally conformally Einstein;
\item each point in $M$ has a neighbourhood $U$ with a function $\vf\in C^\infty(U)$ that solves the PDE (\ref{cepde});
\item each point in $M$ has a neighbourhood $U$ with a closed $1$-form $\Upsilon\in \Gamma(T^*U)$ that satisfies
\belabel{cepde1}
\ro -\nabla \Upsilon +\Upsilon^2 = \lambda \g,\end{equation}
for some function $\lambda$. 
\end{enumerate}
In fact, if $\Upsilon$ solves \cref{cepde1} for some $\lambda$, then $\lambda$ is determined by
\[\lambda=\tfrac{1}{n}\left(\mathsf{J}-\mathrm{div}(\Upsilon)+\g(\Upsilon,\Upsilon)\right).\]

\begin{remark}\label{linPDEremark}
When  substituting $\vf=-\log \sigma$, 
the PDE (\ref{cepde}) is equivalent to the linear PDE for $\sigma$, 
\begin{equation}
\label{celinpde}
\nabla^2\sigma +\sigma \ro=\mu \g,
\end{equation}
for some function $\mu$ determined by taking the trace of the equation. This enables the prolongation of this equation that leads to another equivalence to the local conformal Einstein property: the metric $\hat \g=\sigma^{-2}\g$ is an Einstein metric if and only if  $\sigma $ has no zeros and  satisfies \cref{celinpde}, which in turn is equivalent to 
$\left(\sigma, \d \sigma, -\tfrac{1}{n}(\Delta \s+\s \ro)\right)$ being a parallel section of the {\em normal conformal tractor bundle}, see \cite{bailey-eastwood-gover94,cap-slovak-book1}. Solutions to \cref{celinpde} may however have zeros along sets of measure zero, so their existence is not equivalent to the local conformally Einstein property on all of $M$ but only on a dense open subset, see for example \cite{gover05}.
\end{remark}

Analysing the transformation of the Cotton tensor and the Bach tensor, Gover and Nurowski derived the following  necessary conditions for a metric to be conformally Einstein.
\begin{theorem}[{Gover \& Nurowski \cite[Proposition 2.1]{gover-nurowski04}}]
\label{obstheo}
Let $(M,\g)$ be a semi-Riemannian manifold of dimension $n\ge 3$ that is locally conformally Einstein with Cotton, Bach and Weyl tensors $\A$, $\B$ and $\W$ respectively. Then there is a  $V\in TM$ such that
\belabel{Aobs}
\A-\W(V,.,.,.)=0,
\end{equation}
and 
\belabel{Bobs}
\B-(n-4)\W(V,.,.,V)=0.
\end{equation}
If $(M,\g)$ is not already Einstein, then $V$ does not vanish identically.
\end{theorem}
In this theorem, $V$ corresponds to  the metric dual of the closed $1$-form $\Upsilon$ in \cref{cepde1}, and hence can locally be realised as a gradient vector field. Gover and Nurowski have in fact shown a much stronger result \cite[Theorem 2.2]{gover-nurowski04}: under a certain genericity condition on the Weyl tensor, the \cref{Aobs,Bobs} imply that $V$ locally is a  gradient vector field and hence that $\g$ is locally conformally Einstein. In this article we will not use this stronger version, because the genericity condition is not satisfied for the class of bi-invariant metrics we will consider. Instead, we mainly evaluate the conditions  (\ref{Aobs}) and (\ref{Bobs}) on metric Lie algebras.

  \section{Bi-invariant metrics on Lie groups}
     \label{liesec}
Let $G$ denote a metric Lie group and $\gg$ the corresponding metric Lie algebra with $\ad$-invariant scalar product $\la.,.\ra$, which we will also denote by $\g$ when emphasising that it induces a metric or when using index notation. We denote the Levi-Civita connection of $\g$ by $\nabla$. For elements $X,Y, Z\in \gg$, we have 
\[
\nabla_XY=\tfrac{1}{2}[X,Y],\quad \R(X,Y)Z= -\tfrac{1}{4} \left[\left[X,Y\right],Z\right], \quad \Ric=-\tfrac{1}{4}K,\] where $K$ is the Killing form of $\gg$, and $[X,Y]$ is the Lie bracket on $\gg$.
 All of the curvature tensors are $\ad$-invariant, that is
 \[\ad_X \R=0, \quad \ad_X\Ric=0,\quad \text{ etc.,}\]
 and hence are parallel with respect to $\nabla$. In particular, the scalar curvature is constant and the Cotton tensor  vanishes, $\A=0$.
 Moreover, Lie algebra elements that annihilate one of these tensors form an ideal. For example, for the Weyl tensor $\W$, 
 \[\n=\ann(\W)= \{X\in \gg\mid X\hook \W=0\}\subset \gg \]
 forms an ideal. We call this ideal the {\em Weyl nullity ideal}.
 
 From the conditions in \Cref{obstheo} we get the following necessary conditions for a metric Lie algebra to be conformally Einstein:
\begin{corollary}\label{obsinv}
If a metric Lie algebra $\gg$ 
is  conformally Einstein, then it is Bach flat and, 
if $\gg$ is not Einstein, 
the Weyl nullity ideal is not zero, $\n\not=\{0\}$.
\end{corollary}
\bprf
If $\gg$ is conformally Einstein in the sense of \Cref{cedef}, the corresponding simply connected Lie group $G$ is locally conformally Einstein. This implies that
locally there is a closed one form $\Upsilon$  that satisfies  \cref{cepde1}.  Hence by \Cref{obstheo} the obstructions in (\ref{Aobs}) and (\ref{Bobs}) vanish for a tangent vector $V\in TG$. 
Since $\nabla \ro=0$ the Cotton tensor vanishes and so condition (\ref{Aobs}) reduces to $V\hook \W=0$. Hence, condition (\ref{Bobs}) implies that $\B=0$. Moreover, if $\gg$ is not already Einstein, there is a point in $p\in G$ such that $V|_p\not=0$. Let $0\not= X\in \gg$ such that $X|_p=V|_p$. Then the invariance of $\W$ shows that $X\hook\W=0$, i.e., that there is a nonzero element in the Weyl nullity ideal.
\eprf
    
Before studying the Weyl nullity ideal for metric Lie algebras in the next sections, we analyse Bach flatness. A crucial tensor for this is the square of the Ricci endomorphism, that is, the composition of the $(1,1)$-tensor $\R_i^{~j}$ with itself,  
$\R_{i}^{~k}\R_k^{~j}$,
which we denote for brevity by $\Ric^2$. We also denote its dual, the bilinear form 
$\R_{i}^{~k}\R_{kj}$,
by $\Ric^2$.
If $\Ric^2=0$ we say that $\gg$ has {\em $2$-step nilpotent Ricci tensor}. Note that $2$-step nilpotency of the Ricci tensor cannot occur for Riemannian metrics as it implies that the image of the Ricci endomorphism is totally null, i.e, light-like,
\[0=\R_{i}^{~k}\R_{kj}=\R_{i}^{~k}\R_j^{~l}\g_{kl},\]
where $\g_{kl}$ is the metric defined by the $\ad$-invariant scalar product $\la.,.\ra$.
Moreover, $2$-step nilpotency of the Ricci endomorphism implies that its image is contained in its kernel. In addition, $\Ric^2=0$ implies the vanishing of the scalar curvature (see for example \cite{aipt2} for more details).

The following lemma holds for metric Lie algebras and is crucial for reducing the Bach tensor to the Ricci tensor and its square.

\begin{lemma}
\label{ric2lem}
Let $\gg$ be a metric Lie algebra. Then its Ricci tensor satisfies
\[
\Ric^2 (X,Y)
=-
\tfrac{1}{16}
\tr^\g( K(\ad_{X}(.),\ad_{Y}(.))
=
\tr^\g_{(1,3),(2,6)}(\Ric\otimes \R)(X,Y),\]
where $ K$ is the Killing form  and $\R$ the curvature of $\gg$ defined by the metric $\g=\la.,.\ra$.
Written in index notation, this is
\[
\R_a^{~p}\R_{pb}=-\tfrac{1}{16}\g^{kl}c_{ak}^{\ \ p}c_{bl}^{\ \ q}K_{pq}
=
\R^{pq}\R_{pabq},
\]
where $c_{ak}^{\ \ p}$ are the structure constants of the Lie algebra $\gg$.
\end{lemma}

\bprf
We fix a basis $(X_1, \ldots , X_n) $ of $\gg$, and use indices for the components of tensors in this basis, where  $\R_{ij}$ denotes the components of $\Ric$.
 Then we compute
\begin{eqnarray*}
\Ric^2(X_i,X_j)
&=&
\g^{kl}\R_{ik}\R_{jl}\ =\ -\frac{1}{4}\g^{kl} \g^{pq} K_{ik}\R_{pjlq}
\\
&=&
\frac{1}{16}\g^{kl} \g^{pq} K_{ik}\, \g\left( \left[ \left[ X_j,X_p\right], X_q\right],X_l\right)
\\
&=&
\frac{1}{16}\g^{pq} K\left(X_i,  \left[ \left[ X_j,X_p\right], X_q\right]\right)
\\
&=&
-\frac{1}{16}\g^{pq} K\left( \left[ X_i, X_q\right], \left[ X_j,X_p\right]\right),
\end{eqnarray*}
 by the  $\ad$-invariance of $K$. This implies the first stated equality.
Similarly, we obtain
\begin{eqnarray*}
\R^{kl}\R_{kijl}
&=&
\frac{1}{16}\g^{pk}\g^{ql} K_{pq}\, \g\left( \left[ \left[ X_k,X_i\right], X_j\right],X_l\right)
\\
&=&
\frac{1}{16}\g^{pk} K\left( \left[ \left[ X_k,X_i\right], X_j\right], X_p\right)
\\
&=&
-\frac{1}{16}\g^{pk} K\left( \left[ X_k,X_i\right], \left[ X_p, X_j\right]\right),
\end{eqnarray*}
which gives the second equality.
\eprf

This lemma allows us to compute a formula for the Bach tensor that only involves the Ricci tensor:
\begin{proposition}
\label{bachprop}
The  Bach tensor of a metric Lie algebra of dimension  $n$  satisfies
\[
(n-2)^2\ \B= n\Ric^2 -\frac{n \rho}{n-1}\Ric +\left( \frac{\rho^2}{n-1}-\mathrm{tr}(\Ric^2)\right) \g,\]
where $\rho$ is the scalar curvature of $\gg$ and $\g=\la.,.\ra$ is the metric.
\end{proposition}
\bprf For a metric Lie algebra the Cotton tensor is $\A_{abc}=0$, so that 
 a straightforward computation shows
 \[
 (n-2)^2\,\B_{bc}
 =
 (n-2)\R^{ad}\R_{abcd}+ 2\R_{bd}\R^d_{~c}-\frac{n\, \rho}{n-1}\R_{bc}+\left( \frac{\rho^2}{n-1}-\R_{ad}\R^{ad}\right)\g_{bc}.
 \]
This allows to use \Cref{ric2lem}  to replace contractions $\R^{ad}\R_{abcd}$ of the curvature tensor with $\Ric$ by $\Ric^2$-terms.
\eprf
This formula for the Bach tensor yields a proof of \Cref{bachtheo} from the introduction.

\begin{proof}[{\bf Proof of \Cref{bachtheo}}]
The proof is based on the following observation.
   \begin{lemma}\label{alglemma}
  Let  $A\in \mathrm{End}(\rr^n)$ be a linear map of $\rr^n$
such that 
\begin{equation}
\label{a2eq}
0=A^2+bA+c\1,\quad \text{ for $b,c\in \rr$.}\end{equation} 
Then one of the following two cases occurs: 
  \begin{enumerate}
  \item  $A=\lambda \1 +N$, with some   $N\in\mathrm{End}(\rr^n)$ such that $N^2=0$. When $N\not=0$, then  
  \[
  b=-2\lambda, \qquad c=\lambda^2.\]
  \item $A$ is diagonalisable with only two different eigenvalues $\lambda$ and $\mu$, in which case \[b=-\lambda-\mu,\qquad c=\lambda\mu.\]
  \end{enumerate}
  Conversely, every $A$ as in (1) or (2) satisfies \cref{a2eq}. 
  \end{lemma}
\begin{proof} This lemma follows entirely from the Jordan normal form for $A$.
If $A$ is considered as a linear map on $\mathbb{C}^n$, 
 then squaring a Jordan block of size $3$ or larger shows that relation~(\ref{a2eq}) is not satisfied, which implies that such blocks cannot occur.

 Next assume that 
 $A$ has at least one Jordan block $J_\lambda$ of size $2$. Then the equation for $A$ implies that $b=-2\lambda$ and $c=\lambda^2$, which shows that $\lambda$ is real. 
 In this case any eigenvalue  $\mu$  of $A$, then satisfies 
 \[0=\mu^2+b\mu+c =(\lambda-\mu)^2,\]
 and hence $\lambda=\mu$. This is equivalent to $A=\lambda \1+N$ with $N\not=0$ and $N^2=0$.

 Finally assume that $A$ is diagonalisable. Then \cref{a2eq} for $A$ implies that $A$ can  have at most two  eigenvalues $\lambda$ and $\mu$, in which case $b=-\lambda-\mu$ and $c=\lambda\mu$.  The case of $\lambda=\mu$ is contained in the first alternative.
 
 For the converse it is straightforward to check that every $A$ as in (1) or (2) satisfies \cref{a2eq}.
 \end{proof}
For the proof of \Cref{bachtheo} first we assume that $\gg$ is Bach-flat, use Proposition \ref{bachprop} and apply the lemma to the Ricci tensor $A=\Ric$.  The Bach flatness then implies that 
\[
0
=
\Ric^2 +b\Ric +c \g,\quad\text{ with 
 $b= -\frac{ \rho}{(n-1)}$ and $c=\frac{1}{n}\left( \frac{\rho^2}{n-1}-\mathrm{tr}(\Ric^2)\right)$.}\]
Hence, by  \Cref{alglemma}, we get that either $\Ric$ is diagonalisable with two different eigenvalues or $\Ric=A=\lambda\1+N$ with  $N^2=0$ and 
with  $\rho=\tr(A)=\lambda n$ and $\tr(A^2)=n\lambda^2$.
In the latter case, 
if $N=0$, the Ricci tensor is a multiple of the metric and hence Einstein, so we assume that $N\not=0$. 
The requirements on $b$ and $c$ from \Cref{alglemma} then give that
\[
2\lambda = \frac{n\,\lambda }{(n-1)},\quad
\lambda^2 =\frac{\lambda^2}{n}\left( \frac{n^2}{n-1}-n\right)=
\frac{\lambda^2}{(n-1)}
.\]
Since $n>2 $ was assumed, this can only hold if $\lambda=0$. Hence $\Ric=N$ is $2$-step nilpotent.

In the case when $\Ric=A$ is diagonalisable with two eigenvalues $\lambda\not=\mu$ we have that $\rho=k\lambda+(n-k) \mu$, $\tr(A^2)=k\lambda^2+(n-k)\mu^2$, where $k$, $1\le k<n$, is the dimension of the eigenspace of $\lambda$ and $n-k$ the dimension of the eigen space of $\mu$. With $b$ and $c$ from \Cref{alglemma} we get
\[
\lambda+\mu=\frac{ k\lambda+(n-k)\mu }{(n-1)},
\quad
\lambda\mu= 
\frac{(k\lambda+(n-k)\mu)^2 }{n(n-1)}-\frac{k\lambda^2+(n-k)\mu^2}{n}.\]
The left equation implies that
\belabel{eigen} (k-1) \mu=-(n-k-1)\lambda, \end{equation}
This together with $\lambda+\mu\in \rr$ implies that both eigenvalues are real. Also their eigenspaces $V_\lambda$ and $V_\mu$ are orthogonal to each other:  since $\lambda\not=\mu$, 
\[\lambda \g (v_\lambda,v_\mu) = \g(Av_\lambda ,v_\mu) =\g(v_\lambda ,Av_\mu)=\mu \g(v_\lambda ,v_\mu),\]
shows that  $\g (v_\lambda,v_\mu)=0$. Then the eigenspaces are also complementary, this implies that they are non degenerate with respect to $\g$. 
An important observation is that \Cref{eigen} implies that the kernel of $A$ is either trivial or of dimension $1$.
%
%
If the kernel of $A$ has dimension $1$, it is  spanned by a non-degenerate vector. 
 This proves the only if direction in \Cref{bachtheo}.

For the converse recall  that every semi-Riemannian Einstein manifold is Bach-flat. Moreover, if the Ricci tensor of $\gg$ is $2$-step nilpotent, the scalar curvature vanishes, and by \Cref{bachprop} this implies that the Bach tensor vanishes. 
For the remaining  case in \Cref{bachtheo} with eigenvalues $\lambda\not=0$ and $\mu =-\tfrac{n-k-1}{k-1}\lambda$ one checks  that
$b= -\frac{ \rho}{(n-1)}$ and $c=\frac{1}{n}\left( \frac{\rho^2}{n-1}-\mathrm{tr}(\Ric^2)\right)$ satisfy the conditions on $b$ and $c$ in \Cref{alglemma}, which then implies that $\gg$ is Bach flat by \Cref{bachprop}.
\end{proof}
Finally in this section, let us make a remark about algebraic properties of $\gg$ implying Bach flatness.
It is well known that bi-invariant metrics on nilpotent Lie groups are Ricci-flat. Solvable Lie algebras in general are not Ricci flat, but from Cartan's solvability criterion it follows that the Killing form vanishes on the derived Lie algebra. Hence, from  \Cref{bachprop} and noting that $2$-step nilpotent linear maps have vanishing trace, we obtain the following result (which we could not locate in the literature so far):
\begin{corollary}
\label{solvcor}
Let $\gg$ be a solvable metric Lie algebra. Then $\gg$ has two-step nilpotent Ricci tensor, vanishing scalar curvature, and is Bach flat.
\end{corollary}

  \section{Simple metric Lie algebras}
  \label{simplesec}
  For simple Lie groups, the space of bi-invariant metrics can be described explicitly.
  
  \begin{proposition}
  Let $\gg$ be a real simple  Lie algebra. Then either 
  \begin{enumerate}
  \item The complexification $\gg^\C$ is simple and  the space of $\ad$-invariant symmetric bilinear form is one-dimensional spanned by the Killing form of $\gg$, or 
  \item the complexification $\gg^\C$ is not simple and $\gg=\hh_\rr$ is equal to a complex simple Lie algebra $\hh$ considered as real Lie algebra $\hh_\rr$. In this case the space of $\ad$-invariant symmetric bilinear forms is two-dimensional and spanned by the real and imaginary part of the Killing form of $\hh$.
 \end{enumerate} 
   \end{proposition}
This proposition essentially follows from Schur's lemma. 
It also yields a classification of  $\ad$-invariant symmetric bilinear forms of semi-simple real Lie algebras.
It was noted in   \cite{milnor76}, \cite{medina85} and \cite{albuquerque98}, and can be obtained  from results in \cite{att05} and the fact that simple Lie algebras do not admit ad-invariant skew symmetric bilinear forms.

 Clearly, when 
the complexification $\gg^\C$ is simple,  any bi-invariant metric $\g$ is defined by the Killing form of $\gg$ and hence  Einstein.
The other case is treated in the following theorem.
  \begin{theorem}\label{simpletheo}  Let  $\gg$ be a real simple metric Lie algebra of dimension $n$ with $\ad$-invariant scalar product $\g$ and assume that   $\gg^\C$ is not simple. Then  $\gg=\hh_\rr$ for a complex simple Lie algebra $\hh$ of dimension~$m=\frac{n}{2}$ and $\g$ is of neutral signature $(m,m)$  given by
     \[\g=\lambda K_R+\mu K_I,\] where $K_R$ and $K_I$ are the bi-invarant symmetric bilinear forms corresponding to the real and imaginary part of the Killing form $K_\hh$ of $\hh$ and $\lambda$ and $\mu $ are real constants. The scalar curvature $\rho$ and  Ricci and Bach tensors of $\g$ are given by
     \[
     \Ric=-\frac{1}{2}K_R,\qquad \rho=-\frac{m\lambda }{\lambda^2+\mu^2},\qquad
     \B= \frac{m\lambda\mu }{4(\lambda^2+\mu^2)^2(2m-1)(m-1)}(\mu K_R-\lambda K_I).\]
     Moreover, $\g$ is conformally Einstein if and only if 
     \begin{enumerate}
     \item 
     $\mu=0$, in which case  $\g$ is Einstein, or when 
     \item $\lambda=0$ and $\hh=\sl_2\C$, in which case $\g$ is conformally flat. 
     \end{enumerate}
  \end{theorem}
  \begin{proof}
  
  Let $\hh_0$ be  the split real form  of $\hh$ such that  $\gg=\hh_\rr=\hh_0\+\mathrm{i}\hh_0$ and let $X_a$, $a=1, \ldots ,m $, be a basis of $\hh_0$ and of $\hh$. Denote by 
  $K_{ab}=K_\hh(X_a,X_b)$ the matrix of the Killing form $K_\hh$ of $\hh$ in this basis.
  Then   in the basis $X_a, \mathrm{i} X_b$ of $\gg=\hh_\rr$, the bilinear forms 
  $K_R=\mathrm{Re}\circ K_\hh$ and $K_I=\mathrm{Im}\circ K_\hh$  considered as bilinear forms on $\gg=\hh_\rr$ are of the form
\[K_R=\begin{pmatrix}K_{ab}&0\\0&-K_{ab}\end{pmatrix},\qquad K_I=\begin{pmatrix}0&K_{ab}\\K_{ab}&0\end{pmatrix}.
\]
Note that $K_R=\tfrac{1}{2}K_\gg=-2\Ric$.
The formulae for the scalar curvature and the Bach tensor are obtained from  \Cref{bachprop} by direct computation using that the metric and its inverse are 
\[\begin{pmatrix}
\lambda K_{ab}&\mu K_{ab}
\\
\mu K_{ab}&-\lambda K_{ab}
\end{pmatrix}
,
\quad
\tfrac{1}{\lambda^2+\mu^2}
\begin{pmatrix}
\lambda K^{ab}&\mu K^{ab}
\\
\mu K^{ab}&-\lambda K^{ab}
\end{pmatrix},
\]
where $K^{ab}$ is the inverse matrix of $K_{ab}$.

  Now assume that $\gg$ is conformally Einstein and apply \Cref{obsinv}. Then $\gg$ is  Bach flat only if $\lambda\mu=0$. In the case $\mu=0$ the metric is Einstein, so we assume $\mu\not=0$ and $\lambda=0$, i.e., $\la.,.\ra=\mu K_I$, which is not Einstein, but has vanishing scalar curvature. Moreover, since $\gg=\hh_\rr$ and as there is no complex simple Lie algebra of dimension $2$, the dimension of $\gg$ is even but strictly greater than $4$. Hence, both conditions for the second obstruction in 
\Cref{obsinv} to vanish are satisfied and we conclude that the Weyl tensor $\W$ has a non-trivial  Weyl nullity ideal. Since $\W$ is $\ad$-invariant, its kernel is also $\ad$-invariant, and hence, with $\gg$ being simple, we have $\ker(\W)=\gg$, i.e., $\W\equiv 0$. 
A computation of the Weyl tensor of the metric $\mu K_I$ on $\gg=\hh_\rr$ yields
\[
4\, \W(X,Y,Z,W)= \tfrac{1}{m-1}   \mathrm{Im} \Big( K_\hh(Y,Z) K_\hh(X,W)-K_\hh(X,Z) K_\hh(Y,W)\Big)-K_I\left( \left[ \left[ X,Y\right],Z\right],W\right)\]
 for all $X,Y,Z,W \in \hh$. As $K_\hh$ is a complex bilinear form, this shows that $\W\equiv 0$ implies that 
 \begin{equation}
 \label{weyleq}
  K_\hh\left( \left[ \left[ X,Y\right],Z\right],W\right)= \tfrac{1}{m-1}    \big( K_\hh(Y,Z) K_\hh(X,W)-K_\hh(X,Z) K_\hh(Y,W)\big).\end{equation}
  If the rank of the complex Lie algebra $\hh$ is $\ge 2$, this gives a contradiction:   in \cref{weyleq}, when taking $X$ and $Y=Z$ from a Cartan subalgebra of $\hh$ such that $K_\hh(Y,Y)\not=0$ and  $K_\hh(X,Y)=0$, the left hand side vanishes  (as the Cartan subalgebra is abelian) and thus
  \[0=    K_\hh(Y,Y) K_\hh(X,W),
  \]
  for all $W\in \hh$, which contradicts the non-degeneracy of the Killing form. 
  
  When the rank of $\hh$ is $1$, i.e., when   $\hh=\sl_2\C$, this argument breaks down and  it can be checked directly that equation \cref{weyleq} is indeed  satisfied. Taking the imaginary part of this equation, then shows that for $\hh=\sl_2\C$ the metric  $\g=\mu K_I$ on $\gg=\hh_\rr$ has $\W=0$, i.e., is conformally flat.
  \end{proof}
  \section{Metric Lie algebras via double extensions}
  \subsection{Double extensions}
 \label{desec}
  
  Let $\hh$ be metric Lie algebra with $\ad$-invariant scalar product $ \la.,.\ra_\hh$ and let $\ss$ be Lie algebra. Moreover, let $\delta:\ss\to \mathfrak{der}(\hh)\cap \so(\hh)$ be a Lie algebra homomorphism into the skew symmetric  derivations of $\hh$, that is, $\la \delta_S (X),Y\ra_\hh =-\la X,\delta_S (Y)\ra_\hh$ where we denote by $\delta_S$ the image of $S\in \ss$ under $\delta$. Also, for $X\in \hh$ we denote by $\delta_\bullet (X)$ the linear map from $\ss$ to $\hh$ that sends $S$ to $\delta_S(X)$. 
 
 In this setting the first step in defining a double extension of $\hh$ is to define the {\em central extension} of $\hh$ by $\ss^*$, where $\ss^*$ is the dual vector space to $\ss$, that is given by the cocycle in $Z^2(\hh,\ss^*)$ defined by
 \[\hh\times \hh \ni (X,Y)\longmapsto \la \delta_\bullet (X),Y\ra_\hh\in \ss^*.\]
We denote this central extension by $\ss^*\+_\delta \hh$.   Recalling the definition of a central extension given by a cocycle, the Lie bracket of $\ss^*\+_\delta \hh$ is given by
\[ \left[\begin{pmatrix}\sigma\\ X\end{pmatrix}
,\begin{pmatrix} \hat\sigma \\\hat X\end{pmatrix}\right]=\begin{pmatrix}    \la \delta_\bullet (X), \hat X \ra_\hh \\ \big[X,\hat X\big]\end{pmatrix}\ \in \ \ss^*\oplus \hh.\]

 Next,   we consider the adjoint representation  $\ad $ of $\ss$ and its dual $\ad^*$, the co-adjoint representation of $\ss $ on $\ss^*$  given by $\ad^*_S(\sigma):=\sigma\circ \ad_S\in \ss^*$. Similarly, we denote by $\ad^*_\bullet(\sigma)$ the map that sends $S\in \ss$ to $\ad^*_S(\sigma)\in \ss^*$. This allows us to
extend the map $\delta:\ss\to \der(\hh)$ to a map from $\ss$ to $\der(\ss^*\+_\delta\hh)$, which we also denote by $\delta$,
 \begin{equation}
 \label{delta}
 \delta_S \begin{pmatrix} \sigma\\ X \end{pmatrix}=\begin{pmatrix}  \ad_S^*(\sigma) \\ \delta_S(X)
 \end{pmatrix}.
 \end{equation}
 
 \begin{definition}\label{dedef}
 Let $\hh$ be metric Lie algebra with $\ad$-invariant scalar product $\la.,.\ra_\hh$, let $\ss$ be Lie algebra with an $\ad$-invariant bilinear form $\mathrm{b}$ and let  $\delta:\ss\to \mathfrak{der}(\hh)\cap \so(\hh)$ be a Lie algebra homomorphism into the skew derivations of $\hh$. 
  Then the {\em double extension $\gg$ of $\hh$ by $\ss$ and $\delta$},   is the metric Lie algebra that is given by the semidirect sum of the central extension $\ss^*\+_\delta\hh$ with  $\ss$ and the the map $\delta$ in \cref{delta}, 
   \[
   \gg=(\ss^*\oplus_\delta \hh) \rtimes_\delta \ss,
   \]
together with the  $\ad$-invariant inner product
   $\tg$ given by
   \[
\big\la  \begin{pmatrix} \sigma \\ X \\ S \end{pmatrix}  , \begin{pmatrix} \hat\sigma \\ \hat X\\  \hat S \end {pmatrix}\big\ra_{ \b} 
   =
   \la X,\hat X\ra_\hh+\mathrm{b}(S,\hat S) +\sigma(\hat S)+\hat \sigma (S).\]
 \end{definition}
 
 Recalling the definition of the semidirect sum $\rtimes_\delta$,  the Lie bracket in $\gg$ is given in the splitting 
$\ss^*\oplus \hh \+ \ss$ by
\begin{equation}\label{debracket}
\left[\begin{pmatrix}\sigma \\ X \\ S\end{pmatrix}, \begin{pmatrix}\hat \sigma \\ \hat X \\ \hat S\end{pmatrix}\right] =
\begin{pmatrix}\ad_S^*(\hat \sigma)- \ad_{\hat S}^*( \sigma)+\la\delta_\bullet (X),\hat X\ra_\hh \\[1mm] [X,\hat X]+ \delta_S(\hat{X})-\delta_{\hat S}(X) \\[1mm] [S,\hat S]\end{pmatrix},
\end{equation}
or in terms of the adjoint representation
   \begin{equation}
   \label{dead}
   \ad_{(\sigma,X,S)}=\begin{pmatrix}
   \ad_S^* & \la \delta_\bullet(X),.\ra_\hh &-\ad^*_\bullet (\sigma) \\
   0& \ad^\hh_X+\delta_S &-\delta_\bullet (X) \\
   0&0&\ad_S
   \end{pmatrix},\end{equation}
   for $\sigma\in \ss^*$, $X\in \hh$ and $S\in \ss$, and $\ad^\hh$ the adjoint representation of $\hh$. In particular,
   $\ss^*$ is an abelian ideal and $(\ss^*)^\perp=\ss^*\+_\delta \hh$ is an ideal in $\gg$, and that $\ss$ is a subalgebra of $\gg$.    
   A double extension admits several exact sequences of Lie algebras,  for $\gg$
   \begin{eqnarray}
    \label{exactg1}
   0\ \longrightarrow \  \ss^*\ \hookrightarrow & \gg &\stackrel{p}{\twoheadrightarrow}\  \hh\rtimes_\delta\ss\simeq \gg/(\ss)^*\  \longrightarrow\  0
\\
   \label{exactg}
   0\ \longrightarrow\   (\ss^*)^\perp \ \hookrightarrow & \gg & \stackrel{q}{\twoheadrightarrow}\ \ss\simeq \gg/(\ss^*)^\perp \ 
    \longrightarrow\  0
   \end{eqnarray}
and one for $(\ss^*)^\perp$,
   \begin{equation}
   \label{exacts}
   0\ \longrightarrow\   \ss^* \hookrightarrow\  (\ss^*)^\perp\  \stackrel{r}{\twoheadrightarrow}\  \hh\simeq(\ss^*)^\perp/\ss^* \ \longrightarrow\  0.
   \end{equation}

    The importance of double extensions stems from the  remarkable structure result for indecomposable metric Lie algebras by 
    Medina \& Revoy \cite{medina-revoy85} (see \Cref{mrtheo} in our Introduction), which states that every nonabelian, nonsimple, indecomposable metric Lie algebra is a double extension by a simple or a $1$-dimensional Lie algebra. An interesting algebraic fact that was already observed in \cite{baum-kath03} is the following:    
    \begin{lemma}\label{abellemma}
Let  $\g=(\ss^*\oplus_\delta \hh) \rtimes_\delta \ss$ be a double extension by an abelian Lie algebra $\ss$. Then the metric Lie algebras 
$(\gg,\la.,. \ra_{\b})$ and $(\gg,\la., . \ra_{0})$ are isomorphic as metric Lie algebras.
    \end{lemma}
    \bprf
    Let $S_1,\ldots , S_r$ be a basis of $\ss$, $\b_{ij}=\b(S_i,S_j)$, and $\sigma^i$ a dual basis to $S_i$. Then the vector space isomorphism $F:\gg\to \gg$ defined  by
   \[
 F|_{\ss^*\+\hh }=\mathrm{Id},\quad F(S_i)=-\b_{ik}\sigma^k+S_i,\]
   is an isometry between 
   $(\gg,\la.,.\ra_{\gg,\b})$ and $(\gg,\la.,.\ra_{\gg,0})$, that is, $F^*\la.,.\ra_{\gg,\b}=\la.,.\ra_{\gg,0}$.  That $F$ is also a Lie algebra isomorphism can be easily checked using the assumption that $\ss$ is abelian.   
        \eprf
        
 \begin{remark}
 The computation that is used to show that $F$ is a Lie algebra homomorphism breaks down when $\ss$ is not abelian. Indeed, let $c_{~ij}^k$ be the structure constants of $\ss$. The only terms that prevent $F$ from being a Lie algebra homomorphism are
\begin{eqnarray*}
[ F( S_i),F(S_j)]-F([S_i,S_j])
&=&
\b_{ik}\ad_{S_j}^*(\sigma^k)-\b_{jk}\ad_{S_i}^*(\sigma^k)+c_{~ij}^k \b_{kl}\sigma^l
\\
&=&
(\b_{ik}c_{~jl}^k-\b_{jk}c_{~il}^k +\b_{kl}c_{~ij}^k)\sigma^l
\\
&=&
\tfrac{3}{2}(\b_{ik}c_{~jl}^k-\b_{jk}c_{~il}^k)\sigma^l
\\
&=&
3\b_{ik}c_{~jl}^k\sigma^l
\end{eqnarray*}
using \cref{debracket} and the $\ad^{\ss}$-invariance of $\b$, i.e., that $\b_{kl}c_{~ij}^k =\b_{ki}c_{~jl}^k=-\b_{kj}c_{~il}^k$,  in the two last steps.
 \end{remark}
 In the case of abelian $\ss$, the result in \Cref{abellemma} allows us to assume without loss of generality that $\b=0$, in which case we  denote the $\ad^\gg$-invariant inner product by $\la.,.\ra:=\la.,.\ra_{0}$.



  Now, for a metric double extension,  we will provide a formula for the Ricci tensor (equivalently its Killing form, see also \cite{baum-kath03}) and its square.
  Let  $\gg=(\ss^*\oplus_\delta \hh) \rtimes_\delta \ss$ be a metric double extension with invariant scalar product $\la.,.\ra$ and let $\Ric$ be its Ricci tensor. We identify  $\Ric$ with $-\tfrac{1}{4} K$, where $K$ is the Killing form of $\gg$. 
  Multiplying two matrices in (\ref{dead}) and taking their trace shows that 
   \[\ss^*\hook \Ric=0,\]
   and
   \begin{equation}\label{ric}
   \begin{array}{rcl}
   \Ric(X,Y)&=&-\frac{1}{4} K_\hh (X,Y),\quad\text{ for }X,Y\in \hh,
   \\
   \Ric(S,T)&=&
     -\frac{1}{4}\left(  2K_\ss(S,T)  +\tr(\delta_S\circ\delta_T)\right),\quad\text{ for }S,T\in \ss,
   \\
   \Ric(X,S) &=&-\frac{1}{4} \tr(\delta_S\circ \ad^\hh_X) \\
    & =& -\frac{1}{4}\tr(\ad^\hh_X\circ \delta_S ) \ =\ \Ric(S,X), \ \ \text{ for $S\in \ss$ and $X\in \hh$}.
\end{array}
  \end{equation}
  Here $K_\hh$ and $K_\ss$ are the  Killing forms of $\hh$ and $\ss$, and  $\tr$ denotes the trace of a linear map.
 For future reference we define $\eta \in \hh^*\otimes \ss^*$  by
  \[
  \eta(X,S)= \tr(\delta_S\circ \ad^\hh_X)
  =
  -\h^{ij}\h(X,\ad_{\e_i}\circ \delta_S (\e_j)).
  \]
  Here, $\e_i$ is a basis of $\hh$. For brevity, we will also write $\tr(\delta^2)\in \otimes^2\ss^*$ for  $\tr(\delta_S\circ \delta_{\hat S})$ when $S,\hat S\in \ss$. 
For the square of $\Ric$ this implies 
 \begin{equation}\label{ric2}
  \begin{array}{rcl}
   \Ric^2(X,Y)&=&\frac{1}{16} K^2_\hh (X,Y),\quad\text{ for }X,Y\in \hh,
   \\
     \Ric^2(S,T)&=&
     \frac{1}{16} \eta( \eta(.,S)^\sharp,T),\quad\text{ for }S,T\in \ss,
   \\
    \Ric(X,S) &=&\frac{1}{16}
     K_\hh(X,\eta(.,S)^\sharp), \quad \text{ for $S\in \ss$ and $X\in \hh$},
\end{array}
     \end{equation}
     where $\eta(.,S)^\sharp\in \hh$ denotes the dualisation of the one-form $\eta(.,S) \in \hh^*$ with respect to $\la.,.\ra_\hh$, i.e., 
     $\la \eta(.,S)^\sharp, X\ra_\hh = \eta(X,S)$ for all $X\in \hh$.
     Moreover the scalar curvature of $\gg$ is given by
     \[\rho=-\tfrac{1}{4}\tr (K_\hh)=-\frac{1}{4} \h^{ij}K_\hh(\e_i,\e_j).\]
These observations and \Cref{bachtheo} allow us to prove
\Cref{theodextric} in the introduction:

\begin{proof}[{\bf Proof of \Cref{theodextric}}]
Note that $\ss^*\hook\Ric=0$ but    $\ss^*\hook \g\not=0$. This already implies the first equivalence, that $\gg$ is Einstein if and only if it is Ricci flat.

For the second equivalence, that $\gg$ is Bach flat if and only if $\Ric^2=0$,  we use    \Cref{bachtheo}, and we have to exclude the  possibility of a diagonalisable Ricci tensor  For the double extension however, the abelian ideal $\ss^*$ is always in the kernel of $\Ric$, which excludes the case of two nonzero eigen values, so  according to \Cref{bachtheo} we are left with a one dimensional and non-degenerate kernel of $\Ric$.  But this contradicts the fact that  $\ss^*$ is null. Hence the only remaining possibility for Bach flatness in \Cref{bachtheo} is $\Ric^2=0$. 
\end{proof}

For completeness we collect a few observations that are interesting, but not necessarily needed for our main results. Here, in the case where we assume $\ss=\rr$, we fix an $0\not= S\in \ss$ and a $\delta=\delta_S$. 
\begin{proposition}
Let  $\gg=(\ss^*\oplus_\delta \hh) \rtimes_\delta \ss$ be  a double extension. 
\bnum
\item 
If $\gg$ is solvable/nilpotent, then  $\hh$ and $\ss$ are solvable/nilpotent.
\item
If $\hh $ and $\ss$ are solvable, then $\gg$ is solvable.
\item If $\hh$ is  nilpotent, $\ss=\rr$ and $\delta$ is a nilpotent derivation, then $\gg$ is nilpotent.
\enum
\end{proposition}

\bprf
The proof is based on the maps in the two exact sequences (\ref{exactg}) and (\ref{exacts}).
First assume that $\gg$ is solvable/nilpotent. Then $\ss$, as a homomorphic image under the projection $q$ in (\ref{exactg}), is solvable/nilpotent, and the subalgebra $(\ss^*)^\perp$ is solvable/nilpotent. Hence by the projection $r$ in (\ref{exacts}), $\hh$ is solvable/nilpotent. 

For the second point assume that $\ss$ and $\hh=(\ss^*)^\perp/(\ss^*)$ are solvable. Since $\ss^*$ is a central ideal, this implies that $(\ss^*)^\perp$ is a solvable ideal in $\gg$. But $\gg/(\ss^*)^\perp\simeq \ss$ is solvable and consequently $\gg$ is solvable.

For the third point, recall Engel's Theorem that a Lie algebra is nilpotent if and only if all its adjoints $\ad_X$ are nilpotent linear maps. It is easy to  check that the   linear maps $\ad_{(\s,X,S)}$ in \cref{dead} are nilpotent whenever $\hh$ and $\delta$ are nilpotent. 
 \eprf
A counter example to the implication ``$\hh$ and $\ss$ nilpotent implies $\gg$ nilpotent'' is given by the oscillator algebras (see \Cref{oscsubsec}), for which both $\hh$ and $\ss=\rr$ are abelian, but $\gg$ is only solvable but not nilpotent. Here $(\ss^*)^\perp$ is not central in $\gg$.
However, if $\hh$ is nilpotent and $\ss=\rr$, we can describe the Ricci tensor more precisely: 
\begin{proposition}\label{nilpprop}
Let $\gg$ be a double extension with $\hh$ nilpotent, $\ss=\rr$ and a derivation $\delta$. Then $\Ric= \tr(\delta^2)$ and $\Ric^2=0$, and hence $\B=0$.
\end{proposition}
\begin{proof}
The proof is based on the following observation:
\begin{lemma}\label{nilplemma}
Let $\hh$ be a nilpotent Lie algebra and $\delta$ a derivation of $\hh$. Then for all $X\in \hh$, the linear map $\delta\circ \ad_X$ is nilpotent and in particular $\tr(\delta\circ \ad_X)=0$.
\end{lemma}
\begin{proof}
Let $\hh^k=[\hh,\hh^{k-1}]$ be the lower central series of $\hh$ and $\delta$ a derivation, then 
\begin{equation}\label{dereq}
\delta\circ \ad_X=\ad_{\delta(X)}+\ad_X\circ \delta, \quad\text{ i.e., }\left[ \delta, \ad_X\right]=\ad_{\delta(X)}.
\end{equation}
This can be used to show inductively that $\delta$ preserves $\hh^k$ and  that
the image of $(\delta\circ \ad_X)^k$ is contained in $\hh^k$. Hence with $\hh$ nilpotent, the image of $(\delta\circ \ad_X)^k$ eventually becomes zero, which means that $(\delta\circ \ad_X)^k$ is a nilpotent linear map. As a consequence it is trace free.
\end{proof}
Then with $\hh$ being nilpotent we have that $K_\hh=0$ and the Lemma implies that $\eta=0$, and hence $\Ric= \tr(\delta^2)$ and $\Ric^2=0$. This implies $\B=0$.
\end{proof}

\subsection{Double extensions by simple Lie algebras}\label{desimplesec}
Here we will show that double extensions by simple Lie algebras cannot be conformally Einstein.
The key fact we are going to use in the following is the  algebraic version of the  {\em Karpelevich-Mostow Theorem}:
\begin{theorem}[{Karpelevich \cite{karpelevich53}, Mostow \cite{mostow55}, see also \cite[Corollary 1 in \textsection 6]{onishchik04}}]
\label{kmtheo}
Let $f:\ss\to\hat\ss$ be a homomorphism of real semisimple Lie algebras and let $\ss=\k\+\p$ be a Cartan decomposition of $\ss$. Then there is a Cartan decomposition of $\hat \ss=\hat \k+\hat\p $ such that  $f(\k)\subset \hat \k$ and $f(\p)\subset \hat \p$.
\end{theorem}
Recall (e.g. from \cite[Chapter VI]{knapp96},  \cite[Chapter 4]{onishchik-vinberg3} or \cite{onishchik04}) that a {\em Cartan  decomposition} of a real semisimple Lie algebra $\ss$ is a decomposition $\ss=\k\+\p$ such that 
\bnum
\item
$\k$ is a subalgebra, $[\k,\p]\subset \p$ and $[\p,\p]\subset \k$, and
\item the Killing form of $\ss$ is negative definite on $\k$ and positive definite on $\p$, and $\k$ and $\p$ are orthogonal to each other. 
\enum
A Cartan subalgebra $\mathfrak{t}$ of $\ss$ is a subalgebra such that $\mathfrak{t}^\C$ is a Cartan subalgebra of $\ss^\C$ (see \cite[Sections 2.3.1 and 2.3.7] {cap-slovak-book1} or \cite{Knapp97}). Given a Cartan decomposition $\ss=\k\+\p$, there is a {\em stable Cartan subalgebra} $\mathfrak{t}\not=0$, i.e., a Cartan subalgebra such that $\mathfrak{t}=(\t\cap \k)\+(\t\cap \p)$. The dimension of a Cartan subalgebra is the {\em rank of $\ss$}. 
Using these facts, \Cref{kmtheo} implies the following statement that is useful for our purposes:
\begin{corollary}\label{killcor}
Let $\ss \subset \so(t,s)$ be a subalgebra in $\so(t,s)$ that is semisimple. Let $\ss=\k\+\p$ be a Cartan decomposition and $\t$ a stable Cartan subalgebra. Assume that the Killing form of $\ss$ 
satisfies
\[ K_\ss(X,Y)=\lambda \,
\tr (X\cdot Y),\quad\text{ for all $X,Y,Z\in \t$,}\]
with some $\lambda\in \rr$, where  $\tr$ denotes the trace form  in $\so(t,s)$. Then $\lambda>0$.
\end{corollary}
\bprf
Let $\ss=\k\+\p$ be Cartan decomposition of $\ss$ and $\so(t,s)=\hat \k\+\hat\p $ the associated Cartan decomposition of $\so(t,s)$ with $\k=\hat\k\cap \ss$ and $\p=\hat\p\cap \ss$ and let $\t$ and $\hat\t$ two stable Cartan subalgebras of $\ss$ and $\so(t,s)$.
Then, from the above properties of Cartan decompositions and Killing forms, it follows that 
 both Killing forms $K_{\so(t,s)}$ and $K_\ss$ are negative definite on $\t\cap \k=\hat\t\cap \hat\k$ and positive definite on $\t\cap \p=\hat\t\cap \hat\p$. Using the assumption and the relation between 
the Killing form of $\so(t,s)$ and the trace form we get
\[K_\ss(X,Y)=\lambda \tr (X\cdot Y) =\tfrac{\lambda}{s+t-2}K_{\so (t,s)} (X,Y).\]
Taking $X,Y\in \t\cap \k$ or in $\t\cap \p$ implies that $\lambda>0$.
\eprf
Returning to double extensions by simple Lie algebras we get the following result:

\begin{proposition}\label{simplericprop}
Let $\gg=(\ss^*\+_\delta\hh)\rtimes_\delta\ss$ be a  double extension of a metric Lie algebra $\hh$ by a simple Lie algebra $\ss$. Then $\Ric|_{\ss\times \ss}\not=0$. In particular, $\gg$ cannot be Einstein.
\end{proposition}
\begin{proof}
Since $\ss$ is simple, $\delta(\ss)\subset \so(t,s)$ is either trivial or simple. 
From the computation of the Ricci tensor in (\ref{ric}) we have seen that \[-4 \Ric(S,T)=2 K_\ss (S,T) +\tr (\delta_S\circ \delta_T),\quad\text{ for all }S,T\in \ss.\]
Assuming that this vanishes for all $S,T\in \ss$ gives a contradiction: if $\delta(\ss)$ trivial, then $K_\ss=0$, which contradicts  the simplicity of $\ss$, and if $\delta(\ss)$ is simple it is in contradiction with  \Cref{killcor} applied to $\delta(\ss)\subset \so(t,s)$. 

 Finally, assume that the double extension $\gg$ is Einstein. Then by 
 \Cref{theodextric} $\gg$  is Ricci-flat which contradicts $\Ric|_{\ss\times \ss}\not=0$.
  \end{proof}
A version of this  Theorem in the case that $\hh$ is abelian was  stated in \cite[Theorem 4.1]{baum-kath03}.

Next, in order to analyse the second conformal to Einstein condition for double extensions by a simple Lie algebra $\s$, 
we describe ideals in such double extensions. 
For this we use the two projections $p$ and $q$ in the exact sequences (\ref{exactg1}) and (\ref{exactg}). 
The projection $q$ is simply given by $q(\sigma, X, S)=S\in \ss$.   
  \begin{lemma}\label{slemma}
  Let  $(\gg=(\ss^*\oplus_\delta \hh) \rtimes_\delta \ss,\la.,.\ra_{\gg,\b})$ be a metric double extension by a simple Lie algebra $\ss$, and let $\n$ be an ideal in $\gg$. Then $\n$ contains $\ss^*$ or $\n\subset \ss^*\+_\delta \hh$.
  \end{lemma} 
  \bprf
  Let $\n$ be an   ideal in $\gg$. Then, as the projection $q$ in the sequence (\ref{exactg})  is surjective, $q(\n)\subset \gg/(\ss^*\+\hh)\simeq\ss$ is an ideal in $\ss$. Since $\ss$ is simple, this implies that $q(\n)$ is either trivial, in which case $\n\subset \ss^*\+_\delta \hh$ and the lemma is proven, or isomorphic to $\ss$. In the latter case, for an arbitrary element in $\eta = (\sigma, X,S)\in \n$ we get from \cref{debracket} that
  \[
[\eta,\hat\sigma]=\ad^*_S(\hat\sigma)=\hat \sigma \circ [S,.] \in \ss^*\cap \nn,\quad \text{for all $\hat \sigma\in \ss^*\subset \gg$.}\]
Since $\ss$ is simple, it is $\ss=[\ss,\ss]$, and hence this shows that $\ss^*\subset \nn$. \eprf
\begin{lemma}\label{snulllemma}
Let  $(\gg=(\ss^*\oplus_\delta \hh) \rtimes_\delta \ss,\la.,.\ra_{\gg,\b})$ be a metric double extension of $\hh$ by a simple Lie algebra $\ss$. Assume that $\gg$ has vanishing scalar curvature, so the Killing form of $\hh$ is trace free, and let $\n$ be the Weyl nullity ideal. If   $\n$ contains   $ \ss^*$, then 
$\ss=\sl_2\rr$ or $\ss=\so(3)$ and 
\begin{equation}
\label{rk1conditions}
\hh\hook \Ric=0,\text{ and }\quad  \tr(\delta_Y\circ\delta_Z)=\frac{\dim(\hh)}{2}K_\ss(Y,Z),\ \text{ for all }Y,Z\in \ss.
\end{equation}
\end{lemma}
\begin{proof}
Assume that $\ss^*\subset \n$. Hence we have the condition $\sigma\hook\W=0$ for every $\sigma\in \ss^*$. Evaluating this for $X,Y,Z\in \ss$, and assuming that the scalar curvature of $\gg$ vanishes, we get
\begin{eqnarray*}
0&=&- 4 \W(X,Y,Z,\sigma)\\
&=&
\sigma([[X,Y],Z])+
\\
&& + \tfrac{1}{n-2} \left( \sigma(Y)\left(  2K_\ss(X,Z)+\tr(\delta_X\circ \delta_Z)  
\right)
- \tfrac{1}{n-2}\sigma(X)\left(2 K_\ss(Y,Z)+\tr(\delta_Y\circ \delta_Z) 
\right)\right),
\end{eqnarray*}
where $K_\ss$ is Killing form  of $\ss$. 
If the rank of $\ss$ is not $1$, i.e., if the dimension of a Cartan subalgebra $\t$ is greater than $1$,    we take linearly independent $X,Y$ in  $\t$ and get that 
 $0=\W(X,Y,Z,\sigma)$ yields
 \[ 2K_\ss(Y,Z)+\tr(\delta^2)(Y,Z)=0,\quad\text{ for all $Y,Z\in \t$.}\] If $\t$ is a stable Cartan subalgebra we can use \Cref{killcor} and the same argument as in the proof of \Cref{simplericprop} leads to a contradiction.

In the rank $1$ case, we have $\ss=\so(3)$ or $\ss=\sl_2\rr$. The result is then a direct computation using that for these $\ss$ we have
\begin{equation}\label{rk1}[[X,Y],Z]= -\tfrac{1}{2} \left(K_\ss(X,Z)Y-K_\ss(Y,Z)X\right).\end{equation}
Moreover, $\W(H,Y,Z,\sigma)=0$ yields $\Ric(H,Z)=0$ for all $H\in \hh$ and $Z\in \gg$.
\end{proof}

These lemmas enable us to prove \Cref{desimpletheo} in the introduction, which states that double extensions by simple Lie algebras cannot be conformally Einstein unless $\ss=\so(3)$ or $\ss=\sl_2\rr$..

\begin{proof}[{\bf Proof of \Cref{desimpletheo}}]
Let $\hh$ be a metric Lie algebra and $\gg=(\ss^*\+_\delta \hh)\rtimes_\delta\ss$ a double extension with $\ss$ simple.  
If $\gg$ is conformally Einstein, by \Cref{bachtheo}, the 
first condition $\Ric^2=0$ is satisfied, which implies that the scalar curvature vanishes. Moreover, 
by \Cref{simplericprop},  $\gg$ is  not Einstein, so we can use the vanishing of the second obstruction in \Cref{obsinv}. Let $\n$ be the non trivial Weyl nullity ideal. 

We can assume  that the rank of $\ss$ is at least $2$. Then by \Cref{slemma,snulllemma} we have that $\ss^*\not\subset\n$ and hence $\n\subset \ss^*\+_\delta\hh$. 
We consider the projection \[\pr_\hh:\ss^*\+_\delta\hh\ni(\sigma,H)\mapsto H\in  \hh,\] which is a Lie algebra homomorphism, 
and we denote
$ \n_0:=\pr_\hh(\n)$.
Since $\n$ is an ideal for each $S\in \ss$  and $\sigma+ H\in 
\n$ we have 
$[\sigma+H, S]=-\ad_S^*(\sigma)-\delta_S(H)\in \n$, and hence  that $\n_0$ is invariant under all derivations in the image of $\delta$, i.e., under $\delta(\ss)$. 

Using the vanishing of the scalar curvature again, we will 
evaluate the condition~\ref{Aobs}:
For a nonzero element $N=\sigma+ H\in 
\n$,   and any $X,Y, Z\in \gg$ we get
\begin{eqnarray*}
0&=&\W(X,Y,Z,N)\ = \  \W(X,Y,Z,\sigma)+ \W(X,Y,Z,H)
\\
&=&
-\tfrac{1}{4}\la [[X,Y],Z],H\ra -\tfrac{1}{4}\la [[X,Y],Z],\sigma\ra  +\tfrac{1}{n-2}
\left(  \la Y,\sigma \ra \Ric(X,Z)- \la X,\sigma \ra \Ric(Y,Z) \right)
\\
&&{}+\tfrac{1}{n-2}
\left(
\la X,Z\ra \Ric(Y,H)- \la Y,Z\ra \Ric(X,H)+ \la Y,H \ra \Ric( X,Z)- \la X,H \ra \Ric( Y,Z)\right).
\end{eqnarray*}
By setting $Y=\hat \s\in \ss^*$ we get
\[
0=\W(X,\hat\s,Z,N)
=
-\tfrac{1}{n-2} \la\hat\s, Z\ra  \Ric (X,H),
\]
for all $X,Z\in \gg$ and hence we get  $ (\n_0)\hook \Ric=0$.  The above equation for the Weyl nullity ideal  then simplifies to 
\begin{equation}\label{wneq}
\begin{array}{rcl}
0
&=&
-\tfrac{n-2}{4}\la [[X,Y],Z],H\ra  -\tfrac{n-2}{4}\la [[X,Y],Z],\sigma \ra 
\\[2mm]&&{}
+
\la Y,\sigma \ra \Ric(X,Z)-
 \la X,\sigma \ra \Ric(Y,Z)
+ 
\la Y,H \ra \Ric( X,Z)- \la X,H \ra \Ric( Y,Z),
 \end{array}
 \end{equation}
for all $X,Y,Z\in \gg$. If $X,Y,Z \in \ss$, then as $\ss$ is a subalgebra and  $\ss\perp \hh$,   this implies the same equation as in the proof of \Cref{snulllemma},
\[
0
=
\sigma(  [[X,Y],Z])
+ \tfrac{1}{n-2}\left(  
\s (Y) \left( 2 K_\ss (X,Z) +\tr (\delta_X\circ \delta_Z)\right)  -
\s( X)\left(  
 2 K_\ss (Y,Z) +\tr (\delta_Y\circ \delta_Z) \right) 
\right),
\]
for all $X,Y,Z\in \ss$ and for $\sigma\in \pr_{\ss^*}(\n)$, where $ \pr_{\ss^*}:\ss^*\+_\delta\hh\to \ss^*$ is the projection onto $\ss^*$. If the projection of $\n$ onto $\ss^*$ is equal to $\ss^*$, we get a contradiction in the same way as in the proofs of \Cref{snulllemma} and \Cref{simplericprop}, so we may assume that  $ \pr_{\ss^*}(\n)\not= \ss^*$.

Now we consider  $\n^\perp$, the ideal in $\gg$ that is orthogonal to $\n$ with respect to $\la.,.\ra $. 
As we have $ \pr_{\ss^*}(\n)\not= \ss^*$,  the ideal $\n^\perp$ contains a non-trivial subspace $\ss\cap \n^\perp$. Since $\ss$ is a subalgebra in $\gg$, the subspace  $\ss\cap \n^\perp$ is in fact an ideal in $\ss$ and hence equal to $\ss$, because of the simplicity of $\ss$. Now, since $\hh\perp \ss$, the inclusions $\ss\subset \n^\perp$ and $\n\subset \ss^*\+_\delta\hh$ imply that $\n\subset \hh$. But since $\n$ is an ideal, the bracket relation~(\ref{debracket}) then implies that $\delta_S(H)=0$ for all $S\in \ss$ and $H\in \n_0=\n$. 
With this information, and with $\sigma=0$, \cref{wneq}  becomes
\begin{eqnarray*}
0
&=&
-\tfrac{n-2}{4}\la [[S, Y],T],H\ra  
 + \la Y,H \ra_\hh \, \Ric( S,T)
\\
&=& \tfrac{n-2}{4}\la Y, \delta_S(\delta_T(H)) \ra  
 + \la Y,H \ra_\hh \, \Ric( S,T)
\\
&=&
\la Y,H \ra_\hh \, \Ric( S,T),
\end{eqnarray*}
for all $Y\in \hh$ and $S,T\in \ss$. Since $\la.,.\ra_\hh$ is non-degenerate, this  implies that $\Ric( S,T)=0$ for all $S,T\in \ss$, which leads again to a contradiction as in the proof \Cref{simplericprop}.
\end{proof}

The following example shows that in the case when the rank of $\ss$ is $1$ (see \Cref{snulllemma}), the obstructions can vanish without the metric being conformally Einstein.
\begin{example}\label{so3ex} We will now present an example that is not governed by \Cref{desimpletheo}. Here, $\ss$ is of rank $1$, and we describe a double extension for which both obstructions vanish, i.e., which is Bach flat and has a nontrivial Weyl nullity ideal, but which however {\em is not conformally Einstein}.

Let $\ss=\so(3)$, $\hh=\rr^3$ be abelian with the Euclidean standard inner product, and $\delta=\frac{\sqrt{6}}{2}\,\mathrm{Id}_{\so(3)}$ a map from $\ss$ to the derivations of the abelian Lie algebra $\rr^3$ in $\so(3)$. That is, $\delta_S=\frac{\sqrt{6}}{2}\,S$ for all $S\in \so(3)$. Let $\gg=(\ss^*\+_\delta\rr^3)\rtimes_\delta \ss$ be the double extension  of $\hh=\rr^3$ by $\ss=\so(3)$ and $\delta$ with the inner product $\la.,.\ra_0$, i.e., with $\b=0$.
Let $E_i\in \rr^3$, $i=1,2,3$, be the Cartesian standard basis of $\rr^3$, and 
$S_i\in \ss$, $i=1,2,3$, be a basis of $\so(3)$ such that
\[ [S_i,S_j]=S_k, \quad S_i(E_j)=E_k,\]
where $(i,j,k)$ is an even permutation of $(1,2,3)$. Moreover,  let $\sigma^i\in \ss^*$ be the  dual basis to $S_i$, so $\s^i(S_j)=\delta^i_{~j}$. Then $(\sigma^1, \sigma^2, \sigma^3, E_1,E_2, E_2 , S_1, S_2, S_3)$ forms a basis of $\gg$. The non-vanishing Lie brackets in $\gg$ are
\belabel{example-brackets}
\begin{array}{rcccl}
&&\left[S_i,S_j \right] &= &S_k,
\\
&&\left[ E_i,E_j \right] &=&\tfrac{\sqrt{6}}{2}\sigma^k,
\\
\left[\sigma^i,S_j\right]  &=&\left[S_i,\sigma^j\right] &=& \sigma^k,
\\
\left[S_i,E_j\right]  &=& \left[E_i,S_j\right] &=&\tfrac{\sqrt{6}}{2}E_k,
\end{array}
\end{equation}
where $(i,j,k)$ again is an even permutation of $(1,2,3)$. 

We are now going to show that for the simply connected metric Lie group \[G=(\rr^3_{\text{abel}}\times_\delta \rr^3_{\text{eucl}})\rtimes_\delta \mathbf{SU}(2),\]  both obstructions vanish but that  $\gg$ as metric Lie algebra is not locally conformally Einstein, admitting no solution to \cref{cepde1}.

First we notice that the Schouten tensor of $G$ is given by
\begin{equation}
\label{so3schouten}\ro=\tfrac{1}{4}\,  \delta_{ij}\sigma^i\sigma^j =\tfrac{1}{2} \, \mathring{ \ro }= \tfrac{1}{4} \, \mathring{ \g }
\end{equation}
where $\mathring{ \g }$ and $\mathring{ \ro }$ are the round metric and its Schouten tensor on $\S^3\simeq \mathbf{SU}(2)$. In particular, $G$ is scalar flat and Bach flat so the first obstruction vanishes.
Also the second obstruction vanishes, since 
a direct computation as in the proof of \Cref{slemma} using \cref{rk1} shows that $\ss^*$ is contained in the Weyl nullity ideal, $\ss^*\subset \nn$. 

In regards to  \Cref{cepde1}, we are going to show that in fact $\n=\ss^*$. Since $\ss^*\subset \nn$, it is enough to show that there is no nontrivial $V+S\in \rr^3\+\ss$ that annihilates the Weyl tensor. For such $V+S$ and  all $X,Y,Z\in \gg$ the Weyl tensor is
\begin{eqnarray*}
0&=&
4\,\W(X,Y,Z,V+S)\\
&=&
\langle \left[ \left[X,Y\right],Z\right],V+S\rangle
-\tfrac{1}{7}\left( \la X,Z\ra K_\gg (Y,V+S) -  \la Y,Z\ra K_\gg (X,V+S)\right)
\\
&& -\tfrac{1}{7}\left(\la Y,V+S\ra K_\gg(X,Z)-\la X,V+S\ra K_\gg(Y,Z)\right).
\end{eqnarray*}
Now choosing $X=S_i$ and $Y=S_j$ and $Z\in \rr^3$, and recalling that 
$\left[S_k,Z\right]\in \rr^3$ and hence orthogonal to $\ss$,
leads to the condition, 
\[0=\la \left[S_k,Z\right],V+S\ra= \la \left[S_k,Z\right],V\ra,
\]
where $(i,j,k)$ is an even permutation of $(1,2,3)$. But we also have that $\left[\ss,\rr^3\right]=\rr^3$, which implies $V=0$ as $\la.,.\ra$ is non-degenerate on $\rr^3$. Finally, choosing $X=E_i$ and $Y=E_j\in \rr^3$ and $Z=S_l\in \ss$ we have
\[0=\la[\sigma^k,S_l],S\ra,\]
where $(i,j,k)$ is an even permutation of $(1,2,3)$. Since $ [\ss^*,\ss]=\ss^*$, this implies $S=0$, and we can conclude that $\n=\ss^*$.

Next we assume that there is a rescaling function $\vf\in C^\infty(G)$ such that $\mathrm{e}^{2\vf}  \g$ is Einstein. Since  $\n=\ss^*$, the gradient $\nabla \vf$ of $\vf$ is tangent to $\n=\ss^*$, i.e., 
\[
\nabla \vf= a_i\sigma^i ,\]
for some functions $a_i\in C^\infty(G)$, and
where $\s^i\in \ss^*\subset \Gamma(TG)$. Hence, for the differential of $\vf$ we have
\[\d\vf =g (\nabla \vf,.)=a_i\la \sigma^i,.\ra= a_i\sigma^i ,\]
where now $\sigma^i\in \ss^*\subset \Gamma (T^*\mathbf{SU}(2))$.
Note that the $\s^i$'s here are understood as sections of  $T^*G$ as well as vectors in $\ss^*\subset \Gamma(TG)$, but also as $1$-forms on $\S^3$. 
 Then $0=\d^2\vf=\d a_i\wedge \sigma^i +a_i\d \sigma^i$ shows that
\[\d a_i|_{T(\rr^3_{\text{abel}}\times_\delta \rr^3_{\text{eucl}})}   =0.\]
This implies that the $a_i$'s
are actually smooth functions on $\S^3=\mathbf{SU}(2)$ only.
Hence  $\vf\in C^\infty(\S^3)$, with  $\Upsilon=\d\vf =a_i \sigma^i$, for $a_i\in C^\infty(\S^3)$, is a solution to  \cref{cepde1}, with the function $\lambda$ to be given by
\[\lambda=\tfrac{1}{9}\left(\mathsf{J}-\mathrm{div}(\Upsilon)+\g(\Upsilon,\Upsilon)\right)
=
-\tfrac{1}{9}\mathrm{div}(\Upsilon)=0
.\]
So $\Upsilon$ must satisfy the following equation,
\belabel{ex-eq}
\ro -\nabla \Upsilon +\Upsilon^2 = 0.\end{equation}
Note  that, since $\g$ is not Einstein, $\Upsilon=\d\vf\not=0$. 

We obtain the covariant derivatives of  the $\sigma^i$'s, which are  understood as vector fields on $G$ that are elements in the abelian ideal $\ss^*$ of $\gg$, from the bracket relations (\ref{example-brackets}) as
\[\begin{array}{rclrclrcl}
\nabla_{S_1} \sigma^2&=&\tfrac{1}{2}\sigma^3,&\nabla_{S_2} \sigma^3&=&\tfrac{1}{2}\sigma^1,&
\nabla_{S_3} \sigma^1&=&\tfrac{1}{2}\sigma^2, 
\\
\nabla_{S_2} \sigma^1&=&-\tfrac{1}{2}\sigma^3,& \ \nabla_{S_1} \sigma^3&=&-\tfrac{1}{2}\sigma^2,&
\nabla_{S_3} \sigma^2&=&-\tfrac{1}{2}\sigma^1,
\end{array}
\]
and all other covariant derivatives of $\s^i$ being zero. More concisely for $\nabla\s^i\in \Gamma(TG\otimes T^*G)$, we have
\[\nabla\s^i|_{T(  \rr^3_{\text{abel}}\times_\delta \rr^3_{\text{eucl}} ) }=0\] and 
\begin{equation}
\label{so3lc}
\begin{array}{rcccl}
\nabla \sigma^i|_{T\S^3}&=& \sigma^j\wedge  \s^k& =&\mathring{\nabla} \s^i,
\end{array}\end{equation}
where $\mathring{\nabla}$ is the Levi-Civita connection of the round $3$-sphere and $(i,j,k)$ is an even permutation of $(1,2,3)$. Hence, with \Cref{so3schouten}, our crucial \Cref{ex-eq} becomes
\belabel{ex-eq3}
 \mathring{\nabla} \Upsilon -\Upsilon^2 = \tfrac{1}{4}\mathring{\g},\end{equation}
 which is  an equation only on $\S^3$.
 With $\mathring{\ro}=\tfrac{1}{2}\mathring{\g}$, this equation can be rewritten to a version of the conformally Einstein equation (\ref{cepde1}),
\belabel{ex-eq2}
\mathring{\ro} -\mathring{\nabla} \Upsilon +\Upsilon^2 = \tfrac{1}{4}\mathring{\g},
\end{equation}
on $\S^3$ with $\lambda=\frac{1}{4}$.
This 
 implies that $\vf$ is in fact a local rescaling of the round metric $\mathring{\g}$ \ to another Einstein metric $\hat \g=\mathrm{e}^{2\vf}\mathring{\g}$ on $\S^3$. The round metric on the sphere however is locally conformally flat, and whence the metric $\hat \g$ is a locally conformally flat Einstein metric.  Consequently, $\hat \g$ is a metric of constant curvature $\kappa$ with Schouten tensor $\hat \ro=\frac{\kappa}{2}\hat\g= \frac{\kappa\mathrm{e}^{2\vf}}{2} \mathring{\g}$.
 The transformation of the Schouten tensor in \Cref{rotrafo} then yields
\[ \tfrac{\kappa}{2}\mathrm{e}^{2\vf} \mathring{\g}=\tfrac{1}{2}\left( 1- \mathring{\g}(\Upsilon,\Upsilon)\right)\mathring{\g}-\mathring{\nabla}\Upsilon +\Upsilon^2.\]
This together with \Cref{ex-eq3}
implies that 
\[ \mathring{\g}(\Upsilon,\Upsilon)= \tfrac{1}{2}- 
\kappa \mathrm{e}^{2\vf}.\]
 Then, if  $V$  is the metric dual of $\Upsilon=\d \vf$, i.e.~$V$ is the  the gradient  of $\vf$,  we  have 
\[ X(\g(V,V)) = X\left(\tfrac{1}{2}- 
\kappa \mathrm{e}^{2\vf}\right)=
-2\kappa \mathrm{e}^{2\vf} \Upsilon (X),\quad \text{ for all tangent vectors }X.
\] 
On the other hand it holds that
 \[X(\mathring{\g}(V,V))=2\g(\mathring{\nabla}_XV,V)=2\mathring{\g}(\mathring{\nabla}_VV,X),\quad \text{ for all tangent vectors }X,\]
 and therefore that
 \[\nabla_V\Upsilon = - \kappa \mathrm{e}^{2\vf} \Upsilon.\]
When  inserting $V$ into  \Cref{ex-eq3}, this yields 
 \[
-\tfrac{1}{4}\Upsilon=
 \kappa \mathrm{e}^{2\vf} \Upsilon
+
\Upsilon(V)\Upsilon = 
\left( \kappa \mathrm{e}^{2\vf} 
 +\tfrac{1}{2}- 
\kappa \mathrm{e}^{2\vf}\right)\Upsilon
 =\tfrac{1}{2}\Upsilon.\]
Together with  $\Upsilon\not=0$, this leads to a contradiction\footnote{Michael Eastwood showed us a more conceptual way of producing this contradiction, which uses the linearisation trick in \Cref{linPDEremark}:   For an arbitrary constant $c$, the PDE $\mathring{\nabla} \Upsilon - \Upsilon^2 =c \mathring{\g}$ on the unit sphere $\S^3$ turns into the  linear PDE 
$\mathring{\nabla}\mathring{\nabla}\sigma+c\sigma \mathring{\g}=0$ when substituting  $\vf=-\log(\sigma)$.
Prolonging this equation yields  a connection whose parallel sections correspond to  solutions of the linear PDE. Computing its curvature  shows that the connection has nontrivial parallel sections only when   $c=1$.}
and  shows that $\g$ is not conformally Einstein, even though it is Bach flat and has a nontrivial Weyl nullity ideal.

\end{example}

%
%

\subsection{Double extensions by $\rr$}
When $\ss=\rr$, we fix a nonzero vector $S$ in $\ss$  and its dual $\sigma\in\ss^*$, i.e., $\sigma(S)=1$, and denote by $\delta$ the corresponding derivation of $\hh$. Moreover, $\ss^*$ is contained in the centre of the double extension $\gg=(\ss^*\oplus_\delta \hh) \rtimes_\delta \ss$. 
\Cref{abellemma} implies we can assume that $\mathrm{b}=0$ without loss of generality in the definition of the double extension $\gg$. 
We will be able to make further simplifications because of the following result:

\begin{theorem}[\cite{FavreSantharoubane87}, see also \cite{kelli-thesis}]\label{isotheo}
Let $\hh$ be a Lie algebra, let $\delta$ and $\hat \delta$ be two derivations of $\hh$ in $\so(\hh)$, and let $\gg$ and $\hat \gg$ be the double extensions of $\hh$ by $\delta$ and $\hat \delta$ respectively. Then there is an isomorphism $F:\gg\to\hat\gg$ of metric Lie algebras if and only if
there is a $\lambda\in \rr$, an $X\in \hh$, and an isomorphism $f:\hh\to\hh$ of metric Lie algebras, such that
\[ f \hat\delta f^{-1}= \lambda \delta+\ad_X.\]
\end{theorem} 

In the following 
we will  work with a basis $\e_-,\e_1,\ldots, \e_{n-2},\e_+$ of $\gg$ of the form $\e_-=\sigma$, $\e_a$, $a=1, \ldots, n-2$, a basis of $\hh$ such that $\h_{ab}=\la \e_{a},\e_{b}\ra_\hh$ are constants, and $\e_+=S$. 
We will use the following index convention: greek indices will run from $-, 1,\ldots, n-2, +$ whereas latin indices run from only from $1$ to $n-2$. In this basis the metric satisfies 
\[\g_{+-}=\la \e_-,\e_+\ra =1, \quad \g_{\pm a}=\la \e_\pm ,\e_a\ra =0, \quad \g_{ab}=\h_{ab} =\la\e_a,\e_b\ra_\hh.
\]
The Ricci tensor of a double extension by $\rr$, when written in this basis is given by

\begin{equation}\label{ric1dim}
   \R_{\mu\nu}= -\frac{1}{4}\begin{pmatrix}
   0&0&0
   \\
0& (K_\hh)_{ab} & \eta_a\\
 0&  \eta_b& \tr(\delta^2)
 \end{pmatrix},\quad \R_\mu^{~\nu}
 = -\frac{1}{4}\begin{pmatrix}
   0&\eta_a&\tr(\delta^2)
 \\
0& (K_\hh)_{a}^{~b} & \eta^b\\
 0&  0& 0 \end{pmatrix},
  \end{equation}
  where $(K_\hh)_{ab}$ is the  Killing form of $\hh$, and  $\eta_a$ are
  \[
  \eta_a= \tr(\delta \circ \ad^\hh_{\e_a})
  =
  -\h^{bc}\h(\e_a,\ad_{\e_b}\circ \delta (\e_c)).\] 
Hence, the square of the Ricci tensor is
\[
\Ric^2=\R_{\mu\kappa}\R^\kappa_{~\nu}=
 \frac{1}{16}\begin{pmatrix}
   0&(K_\hh)_a^{~c}\eta_c&\eta^c\eta_c
 \\
0& (K_\hh)_{a}^{~c}  (K_\hh)_{c}^{~b} &  (K_\hh)_{c}^{~b} \eta^c\\
 0&  0& 0 \end{pmatrix}.
\]

Assume that the first obstruction  vanishes, i.e., that the manifold is Bach flat. By \Cref{theodextric} this yields $\Ric^2=\R_{\mu\kappa}\R^\kappa_{~\nu}$, which implies that the scalar curvature vanishes,
we will now evaluate the second  condition. That is, we will consider the existence of a nonzero element  $0\not=V=V^\mu\e_\mu=V^-\e_-+V^a\e_a+V^+\e_+$ in the Weyl nullity ideal, 
\begin{equation}\label{obs2}
\begin{array}{rcl}
0=V^\mu \W_{\alpha\beta\gamma\mu}&=&
\tfrac{V^-}{n-2}(\g_{\beta-}\R_{\gamma\alpha} -\g_{\alpha-} \R_{\gamma\beta})
\\
&&+
V^d\left(\R_{\alpha\beta\gamma d}+\tfrac{2}{n-2}(\g_{\alpha[\gamma} \R_{d]\beta}+ \g_{\beta[d}\R_{\gamma]\alpha})\right)
\\
&&
V^+\left(\R_{\alpha\beta\gamma+}+\tfrac{1}{n-2}(\g_{\alpha\gamma} \R_{+\beta}
- \g_{\alpha +} \R_{\gamma\beta}
+ \g_{\beta +}\R_{\gamma\alpha}- \g_{\beta\gamma}\R_{+\alpha})\right), 
\end{array}
\end{equation}
where $\W$ is the Weyl tensor and where we use that $\R_{\alpha\beta\gamma-}=0$ and $\R_{\alpha-}=0$. 
This already provides us with a first solution to $V\hook\W$ in a special situation:
\begin{proposition}\label{obs2lem} If $\gg$ is a double extension of $\hh$  by $\rr$ and $\delta$ such that $\Ric=\tr(\delta^2)$, that is if 
$K_\hh=0$ and $\eta=0$,
then $\e_-\hook \W=0$, so $\e_-$ is  in the Weyl nullity ideal.
\end{proposition}
As an example, \Cref{nilpprop} shows that the assumptions in this proposition are satisfied when $\hh$ is nilpotent.

Hence, from now on we assume that there is at least one index pair $(c,\beta)$ such that $\R_{c\beta}\not=0$.
Setting $\alpha=-$ in \cref{obs2} we get
\begin{equation}
\label{obs0}
0=
\g_{-\gamma}V^d \R_{d\beta}
+V^+(\g_{-\gamma} \R_{+\beta}
-  \R_{\gamma\beta}),
\end{equation}
which, when evaluated for $\gamma=+$ and $\gamma=a$,  implies the conditions
\begin{equation}
\label{obs12}
0= V^d \R_{d\beta}=-\tfrac{1}{4} \eta(\hat V),\qquad 0=V^+\R_{a\beta}, \end{equation} for all $a=1,\ldots n-2, \ \beta=1, \ldots ,n-2, + $.
This first equation implies that the $\hh$-component $\hat V$ of $V$ is in the kernel of $\Ric$.
  Hence, since we have assumed that there is an index pair  $(c,\beta)$ such that $\R_{c\beta}\not=0$, we have that $V^+=0$ and therefore that $V$ is in the kernel of $\Ric$. Now we evaluate the above equations for $\alpha=\gamma=+$ and $\beta=b$,
\[
\begin{array}{rcl}
0&=&
-\tfrac{V^-}{n-2} \R_{+b}
+
V^d\left(\R_{+ b+ d}+\tfrac{1}{n-2} \g_{bd}\R_{++}\right)
\\[2mm]
&=&
-\tfrac{V^-}{n-2} \R_{+b}-
\g( \R (\hat V,\e_+,)\e_+,\e_b)
+\tfrac{\R_{++}}{n-2} \g(\hat V,\e_b).
\end{array}
\]
This shows that the linear map $A:\hh\to\hh$, defined by $A(X)=\R(X,\e_+)\e_+$, which in fact is given by $-\tfrac{1}{4} \delta^2$, satisfies
\[A(\hat V)=\tfrac{1}{n-2} \tr(A)\hat V -\tfrac{V^-}{n-2} \Ric(\e_+)^\sharp,\]
as $\R_{++}=\tr(A)=-\tfrac{1}{4}\tr(\delta^2)$. This means we have
\[\delta^2(\hat V)= 
\tfrac{1}{n-2} \tr(\delta^2)\hat V -\tfrac{V^-}{n-2} \eta^\sharp.\]
Next, we look at $\alpha=+$, $\beta=b$ and $\gamma=c$ and use $V^d\R_{d+}=0$ from above, so
\[
0
=
-\tfrac{V^-}{n-2}  \R_{cb}
+
V^d\left(\R_{+bc d}+\tfrac{1}{n-2} \g_{bd}\R_{c+}\right)
=
-\tfrac{V^-}{n-2}  \R_{cb}+
V^d\R_{+bc d}+\tfrac{1}{n-2} V_{b}\R_{c+},
\]
which can again be rewritten as
\[0=
-\tfrac{V^-}{n-2}  K^\sharp_\hh  +
\delta\circ \ad_{\hat V} +\tfrac{1}{n-2} \hat{V}\otimes \eta.
\]
Here $K_\hh^\sharp$ is the endomorphism obtained by metric dual of the Killing form, or
\[0=
-\tfrac{V^-}{n-2}  K^\sharp_\hh  +
 \ad_{\hat V}\circ \delta +\tfrac{1}{n-2} \eta^\sharp\otimes \hat{V}^\flat.
\]

Note that 
the equation for $\alpha=a$, $\beta=b$ and $\gamma=+$ follows from this by the Jacobi identity
\[0= V^\mu\W_{ab+\mu}+ V^\mu\W_{+ab\mu} -V^\mu\W_{+ba\mu},\] yielding 
\[
0=\ad^\flat_{\delta(V)}- \tfrac{2}{n-2}\hat{V}^\flat\wedge \eta,\]
where $\ad_{\delta(V)}$ is understood to be dualised with $\h$. With this we arrive at a reformulation of the two obstructions for double extension by $\rr$.

\begin{proposition}\label{obsextprop}
Let $\hh$ be a metric Lie algebra of dimension $m$ with Killing form $K_\hh$, a derivation $\delta$,  and corresponding  $1$-form $\eta(X)=\tr(\delta\circ \ad_X)$. Let $\gg=(\rr^*\+_\delta \hh)\rtimes_\delta \rr$ be the metric  Lie algebra obtained by the  double extension of $\hh$ by $\rr$ and $\delta$. If $\gg$ is conformally Einstein, then it holds that: 
\begin{enumerate}[(a)]
\item $\Ric^2=0$, i.e., $K^2_\hh=0$, $\eta^\sharp\hook K_\hh=0$, and $\eta^\sharp$ is a null vector in $\hh$.
\item If $\dim(\gg)>4$ and $\gg$ is not Ricci-flat  (and hence not Einstein), then either $K_\hh=0$ and $\eta=0$ (in which case it is $\e_-\hook\W=0$), or there is  
 $V=V^-\e_-+\hat V\in \rr^*\oplus \hh$ such that
 \begin{eqnarray}
 \label{eq0}\eta(\hat V)&=&0,
 \\
  \label{eq1}
   \delta^2(\hat V)-
\tfrac{1}{m}\left( \tr(\delta^2)\hat V -V^- \eta^\sharp\right)
&=&0,
\\
 \label{eq2}\delta\circ \ad_{\hat V}
-\tfrac{1}{m} \left(V^- K^\sharp_\hh  -
\hat{V}\otimes \eta\right)&=&0,
\\
\label{eq3}
\ad_{\hat V}\circ [.,.]+\tfrac{2}{m}\hat V^\flat\wedge K^\sharp_\hh&=&0,
\end{eqnarray}
where $\ad$ and $[.,.]$ are those of $\hh$.
\end{enumerate}
\end{proposition}
Note that (\ref{eq3}) is just
\[
\ad_{\hat V}[X,Y]+\tfrac{1}{m}\left(\h(\hat V,X)K^\sharp_\hh(Y)-\h(\hat V,Y)K^\sharp_\hh(X)\right)=0.\]
We will use these equations in the next section, where we will also deal with the cases $n\le 4$.
\section{The oscillator algebras and related double extensions}
\label{oscsec}
\subsection{The oscillator algebras}\label{oscsubsec}
Let $\rr^{t,s}$ be a semi-Euclidean vector space of dimension $\ell=t+s$ and signature $(t,s)$, and denote by $\langle.,.\rangle$ the semi-Euclidean standard inner product. We want to doubly extend the abelian metric Lie algebra $\rr^{t,s}$ by $\rr$ and a  linear map $\Phi\in \so(t,s)$.  If $\Phi$ has a kernel, the double extension 
$(\rr\+_\Phi\rr^{t,s}) \ltimes_\Phi\rr$ is isomorphic as metric Lie algebra to the direct sum of the kernel of $\Phi$ with $( \rr\+_\Phi V)\ltimes_\Phi\rr)$, where $V$ is the image of $\Phi$, and hence decomposable. 
So from now on we assume that $\Phi\in \so(t,s)$ is invertible, 
which implies that $\ell$ is even. In this situation, 
the oscillator algebra $\osc_\Phi(t,s)$ of dimension $m=\ell+2$ and signature $(t+1,s+1)$  is then defined as the double extension of  $\rr^{t,s}$ by $\rr$ and by  $\Phi$,  i.e., \[\osc_\Phi(t,s)=(\rr^*\+_{\Phi}\rr^{t,s})\rtimes_\Phi\rr.\] We set $\osc_\Phi(\ell)= \osc_\Phi(0,\ell)\simeq \osc_\Phi(\ell,0)$. The adjoint of $\osc_\Phi(t,s)$ is given by
   \[\ad_{(\rho ,X,r)}=\begin{pmatrix}
0 & \langle \Phi (X),.\rangle &0\\
   0& r\Phi &-\Phi (X) \\
   0&0&0
   \end{pmatrix},\]
   and the $\ad$-invariant scalar product   is
   $\la.,.\ra =2r \rho  +\langle.,.\rangle_{t,s}$ and hence of signature $(t+1,s+1)$. Here we work in a basis $\e_0,\e_1,\ldots, \e_\ell,\e_{\ell+1}$ with $\e_0=(1,0,0)$, $\e_{\ell+1}=(0,0,1)$ and $\e_i$, $i=1,\ldots , \ell$ a basis of $\hh$. We denote by $e_{\mu}^*$ the (algebraically) dual basis.

\begin{remark}
As the only non abelian Lie algebra of dimension $2$ does not admit an $\ad$-invariant scalar product, every indecomposable double extension of dimension $\le 4$ is isomorphic to $\osc_J(2)$ with $J\in \so(2)$.
\end{remark}

  The centre of $\osc_\Phi (t,s)$ is $\rr^*=\rr \cdot \e_0$ and the derived Lie algebra is equal to   $\rr^*\+ \rr^{t,s}$, which is isomorphic to the Heisenberg algebra and hence nilpotent. Therefore:
   \begin{lemma}
   The oscillator algebras
 $\osc_\Phi(t,s)$ are solvable. Their Killing form is given by
 $K((\rho ,X,r),(\sigma,Y,s))=rs \ \tr(\Phi^2)$, i.e., $K=\tr(\Phi^2)\, (\e_{\ell+1}^*)^2$.
  \end{lemma}
  
By \Cref{solvcor} and \Cref{obs2lem}
  this implies that the oscillator algebras satisfy both conformally Einstein conditions:
  \begin{proposition}
  The oscillator algebras are Bach flat and satisfy 
  $V\hook C=0$ for $V=\e_0=(1,0,0)\in \rr^*$ a central element.
  \end{proposition}
  
For the {\em {\bf Proof of \Cref{osctheo}}}, that the oscillator algebras are conformally Einstein,  it remains to check that the vector field 
$\e_0$ on $G$ indeed satisfies  the conditions  in (A) before  \Cref{cedef}:
as $\e_0$ is a parallel vector field on the corresponding simply connected metric Lie group $G$, it is a gradient vector field, and because of the formula for the Killing form,  $\Ric$ and  $\ro$, its metric dual satisfies equation  \cref{cepde1}. Hence the oscillator algebras are conformally Einstein in the sense of \Cref{cedef}.
\hfill $\square$

\bigskip

In the next section we will consider double extensions of  
the oscillator algebras, for which we will need to know their derivations.
\begin{lemma}\label{oscderlemma}
Every derivation $\delta\in \so(\osc_\Phi(t,s))$ of $\osc_\Phi(t,s)$ is of the form
\[\delta=(\Psi,u):=
\begin{pmatrix}0& \langle u,.\rangle &0
\\
0&\Psi & -u
\\
0&0&0
\end{pmatrix},\quad\text{ with $\Psi\in \so(t,s)$ such that $ [\Phi,\Psi]=0$ and $u\in \rr^{t,s}$,}\]
and the corresponding $1$-form $\eta$ is given by
\[\eta(\rho ,X,r)=r\ \tr(\Psi\circ \Phi),\]
i.e., $\eta=\tr(\Psi\circ \Phi)\e_{\ell+1}^*$ and hence $\eta^\sharp=\tr(\Psi\circ \Phi)\e_0$.
\end{lemma}
\begin{proof}
Since $\rr^*$ is the center of $\osc_\Phi(t,s)$, it is left invariant under $\delta$. This and  
the condition that $\delta $ is skew with respect to $\h$, imply that 
\[\delta=
\begin{pmatrix}a& \langle u,.\rangle &0
\\
0&\Psi & -u
\\
0&0&-a
\end{pmatrix},
\quad\text{
with $\Psi\in \so(t,s)$, $a\in \rr$,  and $u\in \rr^{t,s}$. }\]The conditions that $\delta$ is a derivation means that
\[\left[ \delta,\ad_{(\rho,X,r)}\right]=\ad_{\delta (\rho,X,r)}.\]
Multiplying the corresponding matrices yields the equations
\[a\Phi(X)=0,\quad -ar\Phi=r[\Psi,\Phi],\]
and hence  $a=0$ and $ [\Phi,\Psi]=0$.
%
The formula for $\eta$ is derived by direct calculation.
\end{proof}

\subsection{Double extensions of the oscillator algebras}
First we use \Cref{isotheo} to simplify the derivation we are using for the double extension of the oscillator algebra:
\begin{lemma}
Let $\hh=\osc_\Phi(t,s)$ be the oscillator algebra given by $\Phi\in \so(t,s)$. Moreover, let $\Psi\in \so(t,s)$ and $\delta=(\Psi,0)$ and $\hat\delta=(\Psi, u)$ with $u\in \rr^{t,s}$ two derivations of $\hh$ (with the notations as in \Cref{oscderlemma}). Then the double extensions of $\hh$ by $\delta$ and $\hat \delta$ are isomorphic metric Lie algebras.
\end{lemma}
\bprf
This follows from \Cref{isotheo} with $f=\mathrm{Id}$, $\lambda=1$ and $X=(0,\Psi^{-1}(u),0)^\top\in \hh$.
\eprf

Hence, without loss of generality, we can assume that $\delta=(\Psi,0)$.
Let $\gg_{\Psi,\Phi}$ be the  metric Lie algebra which 
is a $1$-dimensional double extension of an oscillator algebra $\hh=\osc_\Phi(t,s)$ by the derivation $\delta=(\Psi,0)$, that is 
\[\gg_{\Psi,\Phi}= \left(\rr^*\+_{(\Psi,0)}\osc_\Phi(t,s)\right)\rtimes_{(\Psi,0)}\rr.\] The ad-invariant scalar product $\la.,.\ra $ is of signature $(t+2,s+2)$ and $\gg_{\Psi,\Phi}$ is of dimension $n=m+2=\ell+4=t+s+4>4$. 
Since $\delta(\e_0)=0$, the centre of $\gg_{\Psi,\Phi}$ is spanned by $\e_-$ as defined in the previous section and by $\e_0=(1,0,0)\in \osc_{\Phi}(t,s)$. They correspond to two parallel vector fields on $\gg_{\Psi,\Phi}$.
Moreover \Cref{oscderlemma}, or the fact that double extensions of oscillator algebras by $\rr$  are solvable, imply:
\begin{lemma}
The metric Lie algebras  $\gg_{\Psi,\Phi}$  have $2$-step nilpotent Ricci tensor and hence are Bach flat.
\end{lemma}

Analysing the second obstruction by using \Cref{obsextprop} gives the following result:
\begin{theorem}\label{oscexttheo}
Let $\gg_{\Psi,\Phi}$ be the metric Lie algebra that is obtained from an oscillator algebra $\mathfrak{osc}_\Phi(t,s)$ by $1$-dimensional double extension by the derivation $(\Psi,v)$. If  $\gg_{\Psi,\Phi}$ is conformally Einstein, then  either $\tr( \Psi\circ \Phi)=0$ and $\tr(\Phi^2)=0$, or there is a nonzero vector in $W\in\rr^{t,s}$ such that 
\[
\Psi^2 (W) =\tfrac{1}{\ell+2} \tr(\Psi^2) W,\quad
\Phi^2 (W) =\tfrac{1}{\ell+2} \tr(\Phi^2) W,\quad
\Psi\circ\Phi  (W) = \tfrac{1}{\ell+2} \tr(\Psi\circ \Phi),
\] for $\ell=t+s$. Moreover, let $B$ be the
 trace form  of $\so(t,s)$, which is non-degenerate. 
Then the matrices $\Phi$ and $\Psi $ are either a multiple of each other or  span a plane in $\so(t,s)$ that is degenerate with respect to  $B$.
\end{theorem}

\bprf We assume that at least one of $\tr(\Phi^2)=0$ or $\tr(\Psi\circ \Phi)$ is not zero, so one of $K_\hh$ or $\eta$ is not zero for $\hh=\osc_\Phi(t,s)$.
Then, by the virtue of \Cref{obsextprop},  conformally Einstein implies  the  equations (\ref{eq0}--\ref{eq3}) for 
 $\hat V= \sum_{i=0}^{\ell+1}V^i\e_i$.
 Since $\eta=\tr(\Psi \circ \Phi)\e_{\ell+1}^*$,  equation \ref{eq3} evaluated for the pair  $\e_0$ and $\e_{\ell+1}$  and \ref{eq0} imply that
\[ V^{\ell+1}\tr(\Phi^2)=V^{\ell+1}\tr(\Psi \circ \Phi) =0,\]
and hence $V^{\ell+1}=0$. Then $\hat V=V^0\e_0+W$ with $W=\sum_{i=1}^\ell V^i\e_i \in \rr^{t,s}\subset \osc_\Phi(t,s)$. 
Equation (\ref{eq1}) gives two more equations:
\begin{eqnarray}
\label{eq1a}
\Psi^2 (W) -\tfrac{1}{\ell+2} \tr(\Psi^2) W&=&0,
\\
\label{eq1b}
V^0\tr(\Psi^2)-V^-\tr(\Psi\circ \Phi)&=&0.
\end{eqnarray} 
 In the same way equation (\ref{eq2}) yields  equations,
 \begin{eqnarray}
\label{eq2a}
\Psi\circ\Phi  (W) - \tfrac{1}{\ell+2} \tr(\Psi\circ \Phi) W&=&0,
\\
\label{eq2b}
V^0\tr(\Psi\circ \Phi)-V^-\tr( \Phi^2)
&=&0.
\end{eqnarray}
Finally, for the oscillator algebras equation (\ref{eq3}) reduces to
\[\ad_W\circ [.,.] -\tfrac{2}{\ell+2}\tr(\Phi^2) W^\flat\wedge \e_{\ell+1}^*\otimes \e_0=0.\]
The only pair of arguments for which this equation gives a condition is $\e_{\ell+1}$ and $X\in \rr^{t,s}$, for which we obtain
\begin{eqnarray*}
0
&=&
[W,\Phi X]+ \tfrac{1}{\ell+2}\tr(\Phi^2)\langle W,X\rangle \e_0\\
& =&
\left( \langle \Phi W,\Phi X\rangle  + \tfrac{1}{\ell+2}\tr(\Phi^2)
\langle W,X\rangle \right) \e_0\\
&
=&
\left(- \langle \Phi^2 W,X \rangle + \tfrac{1}{\ell+2}\tr(\Phi^2)
\langle W,X\rangle \right) \e_0,
\end{eqnarray*}
for all $X\in \rr^{t,s}$. This yields 
the equation
 \begin{eqnarray}
\label{eq3a}
\Phi^2 (W) -\tfrac{1}{\ell+2} \tr(\Phi^2) W&=&0.
\end{eqnarray}
Equations 
 (\ref{eq1b}) and (\ref{eq2b}) can be concisely written as
 \[
\underbrace{ \begin{pmatrix}
 \tr(\Phi^2)&- \tr(\Psi\circ\Phi) 
 \\
- \tr(\Psi\circ\Phi) & \tr(\Psi^2)
 \end{pmatrix}}_{=:B(\Psi,\Phi)}
 \begin{pmatrix}
 V^-
 \\
 V^0
 \end{pmatrix}
 =
 0.
\]
This system has a non-trivial solutions space if and only if 
\[0=\det (B(\Psi,\Phi))=  \tr(\Phi^2) \tr(\Psi^2) - (\tr(\Psi\circ\Phi))^2.\]
Now observe that $[\Psi,\Phi]=0$ implies that 
$(\Psi\circ \Phi)^2 W = \Psi^2\Phi^2 W$, so that 
equations (\ref{eq1a}), (\ref{eq2a}) and (\ref{eq3a}) ensure {\em in a remarkable way} that $\det (B(\Psi,\Phi))=0$, so that equations  (\ref{eq1b}) and (\ref{eq2b}) have indeed a non-trivial solution $(V^-,V^0)$.
\eprf

Now recall that the trace form $B$ of $\so(\ell)$ is negative definite. Moreover,  the indecomposability of $\gg_{\Psi,\Phi}$ implies that $\Psi$ and $\Phi$ are not multiples of each other (see \cite[Proposition 7.1]{baum-kath03} and \cite{kath-olbrich04}). Therefore,  \Cref{oscexttheo} yields the following conclusion:
\begin{corollary}\label{oscextcor}
If the metric Lie algebra $\gg_{\Psi,\Phi}$ is of signature $(2,\ell+2)$ and conformally Einstein, then $\Psi $ is a multiple of $\Phi$. In particular, if $\gg_{\Psi,\Phi}$ is indecomposable, it is not conformally Einstein.
\end{corollary}

\subsection{The remaining case in signature $(2,n-2)$}
In \cite{baum-kath03} it was shown that every indecomposable, nonsimple metric Lie algebra $\gg$ in signature $(2,n-2)$ is isomorphic to one of the following cases: 
\begin{enumerate}
\item $\gg=\osc_\Phi(1,n-3)$, which is conformally Einstein by our \Cref{osctheo};
\item to a double extension of an oscillator algebra, i.e.,  \[\gg=\gg_{\Psi,\Phi} =(\rr^*\+_\Psi  \osc_\Phi(n-4))\rtimes_\Psi\rr;\] we have seen in the previous section that they are   not conformally Einstein whenever they are indecomposable;
\item or to a double extension of the direct sum of $\osc_\Phi(n-5)$ with $\rr$, i.e., 
\[
\gg=(\rr^*\+_{\delta } \hh) \rtimes_\delta \rr\quad\text{ with } \hh = \osc_\Phi(n-5)\+\rr, \]
with $\Phi\in \so(n-5)$ with $\Phi\not=0$.
\end{enumerate}
In the remainder we will show that the last case (3) is not conformally Einstein.  Since $\hh= \osc_\Phi(n-5) \+\rr $ is solvable, its double extension is solvable and hence Bach flat, so we will focus on the second criterion, the non trivial Weyl nullity ideal, to show that $\gg$ is not conformally Einstein. These are the conditions in (b) of \Cref{obsextprop}.
For this we fix a basis $\e_0,\ldots , \e_{\ell+2}$ of $\hh$, where $\e_0,\ldots , \e_{\ell+1}$ is a basis of the oscillator algebra (as in the previous section), $\ell = n-5$, and $\e_{\ell+2}$ spans the central $\rr$ direction. 
From \Cref{obsextprop} we get  the existence of a $\hat V=V^0\e_0 + W+ V^{\ell+1}\e_{\ell+1}+V^{\ell+2}\e_{\ell+2}$ with $W\in \rr^{n-5}$,  that satisfies equations (\ref{eq0}--\ref{eq3}).

The  Killing form of $\hh$ is again given by $K_\hh=tr(\Phi^2)\, (\e_{\ell+1}^*)^2$. 
Since $\e_0$ and $\e_{\ell+2}$ are in the centre of $\hh$ and since $\e_0^\flat= \e_{\ell+1}^*$, $\e_{\ell+1}^\flat=\e_0^*$ and  $\e_{\ell+2}^\flat=-\e_{\ell+2}^*$, equation (\ref{eq3}) applied to the pairs $(\e_0, \e_{\ell+1})$ and $(\e_{\ell+2}, \e_{\ell+1})$ gives that 
\[ \tr(\Phi^2)\, V^{\ell+1}=\tr(\Phi^2)\, V^{\ell+2}=0.\] 
Since  $\Phi\in \so(n-5)$, the vanishing of  $\tr(\Phi^2)$ would imply $\Phi=0$ which is excluded in this case. Hence, as for the oscillator algebra we have that $\hat V=V^0\e_0+W$ with $W\in \rr^{n-5}$.

Next we have to determine the derivations of a Lie algebra of the form 
\[\hh = \rr\+\osc_\Phi(n-5).\]
As for the oscillator algebras, one can show \cite{kelli-thesis} that the derivations of $\hh$ are of the form
\[\delta=
\begin{pmatrix}
0& \langle u,.\rangle &0&-c
\\
0&\Psi & -u&0
\\
0&0&0&0
\\
0&0&c&0
\end{pmatrix},\]
with $\Psi\in \so(t,s)$, $c\in \rr$,   $u\in \rr^{t,s}$ and $\left[\Phi,\Psi\right]=0$. 
Moreover, again using \Cref{isotheo}, one can show that every double extension of $\hh$ by $\rr$ and $\delta$ with  $c\not=0$ is isomorphic to a double extension by 
\[\delta=
\begin{pmatrix}
0& 0 &0&-1
\\
0&\Psi & 0&0
\\
0&0&0&0
\\
0&0&1&0
\end{pmatrix},\]
and hence we can assume this  without loss of generality. As $\gg$ is indecomposable this implies that $\Phi$ and $\Phi$ are not a multiple of each other. The proof for these statements can be found in \cite{kath-olbrich04, kath-olbrich06} or \cite{kelli-thesis}, see also \cite[Theorem 7.1]{baum-kath03}. 

From now on the proof that $\gg$ is not conformally Einstein proceeds with the derivation of equations  (\ref{eq1}--\ref{eq3a}) completely analogous to the proof of \Cref{oscexttheo} and \Cref{oscextcor}, for details see \cite{kelli-thesis}. This yields the following conclusion, which gives a proof of \Cref{2ntheo}:
\begin{theorem}
Let $\gg=(\rr^*\+_{\delta } \hh) \rtimes_\delta \rr$ be an indecomposable  metric Lie algebra that is given
by a double extension of $ \hh = \osc_\Phi(n-5)\+\rr$ by a derivation $\delta$. Then $\gg$ is not conformally Einstein.
\end{theorem}

%
\providecommand{\MR}[1]{}\def\cprime{$'$} \def\cprime{$'$} \def\cprime{$'$}

  \end{document}